\documentclass{amsart}[12pt]
\usepackage[curve,matrix,arrow,frame,tips]{xy}
\usepackage{mathrsfs}
\usepackage{enumerate}
\usepackage{amssymb}
\usepackage{amsthm}
\usepackage[pdftex, pdfborder={0 0 0}, pdftitle={Equivariant properties of symmetric products},pdfauthor={Stefan Schwede},pdfstartview={FitH}]{hyperref}

\DeclareMathOperator{\colim}{colim}
\DeclareMathOperator{\Hom}{Hom}
\DeclareMathOperator{\sa}{sa}
\DeclareMathOperator{\gl}{gl}
\DeclareMathOperator{\map}{map}
\DeclareMathOperator{\Id}{Id}
\DeclareMathOperator{\ev}{ev}
\DeclareMathOperator{\Rep}{\mathbf{Rep}}
\DeclareMathOperator{\tr}{tr}
\DeclareMathOperator{\res}{res}

\newcommand{\spc}{spc}
\newcommand{\spec}{{\mathcal S}p} 
\newcommand{\mA}{{\mathbb A}}
\newcommand{\mF}{{\mathbb F}}
\newcommand{\mQ}{{\mathbb Q}}
\newcommand{\mR}{{\mathbb R}}
\newcommand{\mS}{{\mathbb S}}
\newcommand{\mZ}{{\mathbb Z}}
\newcommand{\Fc}{{\mathcal F}}
\newcommand{\Uc}{{\mathcal U}}
\newcommand{\bA}{{\mathbf A}}
\newcommand{\bB}{{\mathbf B}}
\newcommand{\bL}{{\mathbf L}}
\newcommand{\bO}{{\mathbf O}}
\newcommand{\bRU}{{\mathbf{RU}}}
\newcommand{\iso}{\cong}
\newcommand{\sm}{\wedge}
\newcommand{\tensor}{\otimes}
\newcommand{\xra}{\xrightarrow}
\newcommand{\bs}{\backslash}
\newcommand{\upi}{{\underline \pi}}
\newcommand{\td}[1]{\langle #1\rangle}
\renewcommand{\to}{\longrightarrow}

\numberwithin{equation}{section}
\newtheorem{theorem}[equation]{Theorem}
\newtheorem{lemma}[equation]{Lemma}
\newtheorem{cor}[equation]{Corollary}
\newtheorem{prop}[equation]{Proposition}

\theoremstyle{definition}
\newtheorem{defn}[equation]{Definition}
\newtheorem{rk}[equation]{Remark}
\newtheorem{eg}[equation]{Example}
\newtheorem{construction}[equation]{Construction}

\begin{document}

\title[Equivariant properties of symmetric products]
{Equivariant properties of symmetric products}

\date{\today; 2010 AMS Math.\ Subj.\ Class.: 55P91}
\author{Stefan Schwede}
\address{Mathematisches Institut, Universit\"at Bonn, Germany}
\email{schwede@math.uni-bonn.de}

\begin{abstract}
The filtration of the infinite symmetric product of spheres by the number 
of factors provides a sequence of spectra between the sphere spectrum 
and the integral Eilenberg-Mac Lane spectrum. This filtration has received 
a lot of attention and the subquotients are interesting stable homotopy types.
While the symmetric product filtration has been a major focus 
of research since the 1980s, essentially nothing was known
when one adds group actions into the picture.
We investigate the equivariant stable homotopy types, for compact Lie groups, 
obtained from this filtration of infinite symmetric products of representation spheres. 
The situation differs from the non-equivariant case, for example the subquotients 
of the filtration are no longer rationally trivial and on the zeroth equivariant 
homotopy groups an interesting filtration of the augmentation ideals 
of the Burnside rings arises. Our method is by global homotopy theory, i.e., 
we study the simultaneous behavior for all compact Lie groups at once.
\end{abstract}

\maketitle

\section*{Introduction}

We let $Sp^\infty(X)$ denote the infinite symmetric product of a based space $X$. 
It comes with a filtration by finite symmetric products $Sp^n(X)=X^n/\Sigma_n$. 
We denote by
\[ Sp^n =  \{ Sp^n(S^m)\}_{m\geq 0} \]
the spectrum whose terms are the $n$-th symmetric products of spheres.
A celebrated theorem of
Dold and Thom~\cite{dold-thom} asserts that $Sp^\infty(S^m)$ 
is an Eilenberg-Mac\,Lane space of type $(\mathbb Z,m)$ for $m\geq 1$;
so $Sp^\infty$ is an Eilenberg-Mac\,Lane spectrum for the group $\mathbb Z$. 
The resulting filtration
\[ \mS\ =\ Sp^1\ \subseteq \ Sp^2\ \subseteq\ \dots\ \subseteq 
\ Sp^n\ \subseteq\ \dots \]
of the Eilenberg-Mac Lane spectrum~$Sp^\infty$,
starting with the sphere spectrum~$\mS$, has been much studied.
The subquotient $Sp^n/Sp^{n-1}$ is stably contractible unless
$n$ is a prime power.
If~$n=p^k$ for a prime~$p$, then $Sp^n/Sp^{n-1}$
is $p$-torsion, and its mod-$p$ cohomology 
has been worked out by Nakaoka~\cite{nakaoka-cohomology of symmetric}.
For $p=2$ these subquotient spectra feature in the
work of Mitchell and Priddy on stable splitting 
of classifying spaces~\cite{mitchell-priddy},
and in Kuhn's solution of the Whitehead conjecture~\cite{kuhn:Whitehead}.
Arone and Dwyer relate these spectra to the partition complex, 
the homology of the dual Lie representation
and the Tits building~\cite{arone-dwyer}.

While the symmetric product filtration has been a major focus 
of research since the 1980s, essentially nothing was known
when one adds group actions into the picture.
This paper is about equivariant features of the symmetric product filtration,
for actions of compact Lie groups~$G$.
If~$V$ is a finite dimensional orthogonal $G$-representation, then~$G$ 
acts continuously on the one-point compactification~$S^V$,
and hence on $Sp^n(S^V)$ and~$Sp^{\infty}(S^V)$ by functoriality of
symmetric products.
As~$V$ varies over all such $G$-representations, the $G$-spaces~$Sp^n(S^V)$
form a $G$-spectrum that represents a `genuine' $G$-equivariant
stable homotopy type. For understanding these equivariant homotopy types
it is extremely beneficial not to study one compact Lie group at a time,
but to use the `global' perspective.
Here `global' refers to simultaneous and compatible actions of all compact
Lie groups. Various ways to formalize this idea have been
explored, compare~\cite[Ch.\,II]{lms}, \cite[Sec.~5]{greenlees-may-completion}, 
\cite{bohmann-orthogonal};
we use a different approach via orthogonal spectra.

In Definition~\ref{def-global functor} we introduce the notion of
{\em global functor}, a useful language to describe
the collection of equivariant homotopy groups of an orthogonal spectrum as a whole, 
i.e., when the compact Lie group is allowed to vary.
The category of global functors is a symmetric monoidal abelian category
that takes up the role in global homotopy theory played by abelian groups
in ordinary homotopy theory, or by $G$-Mackey functors
in $G$-equivariant homotopy theory.
As a consequence of Theorem~\ref{thm-Burnside category basis} 
we will see that a global functor is a certain kind of `global Mackey functor' 
that assigns abelian groups to all compact Lie
groups and comes with restriction maps along continuous group homomorphisms
and transfer maps along inclusions of closed subgroups.

In this language, we can then identify the global functor $\upi_0( Sp^n )$
as a quotient of the Burnside ring global functor~$\mA$
by a single relation.
We define an element~$t_n$ in the Burnside ring of the $n$-th symmetric group by
\[ t_n \ = \ n\cdot 1\ - \ \tr_{\Sigma_{n-1}}^{\Sigma_n}(1) 
\ \in \ \mA(\Sigma_n) .\]
As an element in the Grothendieck group of finite $\Sigma_n$-sets,
the class~$t_n$ corresponds to the formal difference of 
a trivial $\Sigma_n$-set with $n$ elements and
the tautological $\Sigma_n$-set $\{1,\dots,n\}$.
Since $t_n$ has zero augmentation, the global subfunctor 
$\langle t_n\rangle$ generated by $t_n$ lies in 
the augmentation ideal global functor~$I$.
The restriction of $t_n$ to the Burnside ring of $\Sigma_{n-1}$ equals $t_{n-1}$,
so we obtain a nested sequence of global functors
\[  
 0 = \langle t_1\rangle \ \subset\ 
\langle t_2\rangle\ \  \subset\ \dots \  \subset\ \langle t_n\rangle \  \subset \ \dots
\   \subset I \ \subset \ \mA\ .  
\]
As part of our main result, Theorem~\ref{thm-pi_0 Sp^n}, 
we prove the following:\medskip

{\bf Theorem.}\
For every $n\geq 1$ the morphism of global functors~$i_*:\mA=\upi_0(\mS)\to\upi_0(Sp^n)$
induced by the embedding $i:\mS= Sp^1\to Sp^n$ passes to an isomorphism 
of global functors
\[ \mA/\td{t_n} \ \iso \ \upi_0(Sp^n) \ .\]  

\medskip

It is then a purely algebraic exercise to describe $\pi^G_0(Sp^n)$ as an
explicit quotient of the Burnside ring~$\mA(G)$:
one has to enumerate all relations in $\mA(G)$ obtained by applying restrictions 
and transfers to the class $t_n$. 
We do this in Proposition~\ref{prop-describe I_n} 
and then work out the examples of~$p$-groups and some symmetric groups.
The author thinks that the explicit answer for $\pi^G_0(Sp^n)$ 
is far less enlightening than the global description of~$\upi_0(Sp^n)$.
Since all the inclusions~$\td{t_{n-1}}\subset\td{t_n}$ are proper,
the subquotients $Sp^n/Sp^{n-1}$ are all globally non-trivial,
in sharp constrast to the non-equivariant situation.

Our calculation of~$\upi_0 (Sp^n)$
is a consequence of a global homotopy pushout square,
see Theorem~\ref{thm-main homotopy}, 
showing that~$Sp^n$ is obtained from~$Sp^{n-1}$ by coning off 
a certain morphism from the suspension spectrum of the
global classifying space of the family of
non-transitive subgroups of $\Sigma_n$.
This homotopy pushout square is a global equivariant refinement of 
a non-equivariant  homotopy pushout established by Lesh~\cite{lesh-filtration}.

Another consequence of our calculations is a possibly unexpected feature of the
equivariant homotopy groups~$\pi_0^G(Sp^\infty)$ when $G$ has positive dimension.
We let $I_\infty$ denote the global subfunctor of~$\mA$ 
generated by all the classes $t_n$ for $n\geq 1$. 
Also in Theorem~\ref{thm-pi_0 Sp^n} we show that
the embedding $\mS\to Sp^\infty$ induces an isomorphism of global functors
\[ \mA/I_\infty \ \cong \ \upi_0 (Sp^\infty) \ .\]  
For every compact Lie group $G$ the restriction map
\[  \res^G_e \ : \ \pi_0^G(Sp^\infty) \ \to \ \pi_0^e(Sp^\infty)\ \iso \ \mZ   \]
to the non-equivariant 0-th homotopy group is a split epimorphism
onto a free abelian group of rank~1.
When the group~$G$ {\em finite}, then this restriction map
is an isomorphism and all $G$-equivariant homotopy groups of
$Sp^\infty$ vanish in dimensions different from~0. 
So through the eyes of finite groups, $Sp^\infty$ is an Eilenberg-Mac\,Lane 
spectrum for the constant global functor~$\underline\mZ$.
This does {\em not}, however, generalize to compact Lie groups
of positive dimension. In that generality, 
the restriction map $\res^G_e$ can have a non-trivial kernel;
equivalently, the value of the global functor $I_\infty$ at some
compact Lie groups is strictly smaller than the augmentation ideal.
We discuss these phenomena in more detail at the end of Section~\ref{sec-examples}.

I would like to thank Markus Hausmann for various helpful suggestions related
to this paper.

\section{Orthogonal spaces}

In this section we recall orthogonal spaces from a global equivariant perspective.
We work in the category~$\mathcal T$ of
{\em compactly generated spaces} in the sense of~\cite{mccord},
i.e., $k$-spaces (also called {\em Kelley spaces})
that satisfy the weak Hausdorff condition.
An {\em inner product space} is a finite dimensional $\mR$-vector space~$V$ 
equipped with a scalar product.
We write $O(V)$ for the orthogonal group of $V$, i.e.,
the Lie group of linear isometries of $V$.
We denote by $\bL$ the category with objects the 
inner product spaces and morphisms the linear isometric embeddings.
This is a topological category as follows: 
if $\varphi:V\to W$ is one linear isometric embedding,
then the action of the orthogonal group $O(W)$, by postcomposition,
induces a bijection
\[ O(W)/O(\varphi^\perp) \ \iso \ \bL(V,W) \ ,\quad 
A\cdot O(\varphi^\perp) \ \longmapsto A\circ \varphi\ ,  \]
where $\varphi^\perp=W-\varphi(V)$ is the orthogonal complement of the image of $\varphi$.
We topologize $\bL(V,W)$ so that this bijection is a homeomorphism, and
this topology is independent of $\varphi$.
So if $n=\dim(V)$, then~$\bL(V,W)$ is homeomorphic to the Stiefel manifold
of orthonormal $n$-frames in~$W$.

\begin{defn} An {\em orthogonal space}
is a continuous functor $Y:\bL\to\mathcal T$ to the category of spaces.
A morphism of orthogonal spaces is a natural transformation.
We denote by $\spc$ the category of orthogonal spaces.
\end{defn}

The systematic use of inner product spaces
to index objects in homotopy theory seems to go back to Boardman's 
thesis~\cite{boardman-thesis}.
The category $\bL$ (or its extension that
also contains countably infinite dimensional inner product spaces)
is denoted $\mathscr I$ by Boardman and Vogt~\cite{boardman-vogt-homotopy everything},
and this notation is also used in~\cite{may-quinn-ray};
other sources use the symbol $\mathcal I$.
Accordingly, orthogonal spaces are sometimes referred to as $\mathscr I$-functors,
$\mathscr I$-spaces or $\mathcal I$-spaces.
Our justification for using yet another name is twofold: on the one hand,
we shift the emphasis away from a focus on non-equivariant homotopy types, 
and towards viewing an
orthogonal space as representing compatible equivariant homotopy types for
all compact Lie groups. Secondly, we want to stress the analogy between
orthogonal spaces and orthogonal spectra, the former being an unstable global
world and the latter a corresponding stable global world.

We let $G$ be a compact Lie group. By a {\em $G$-representation}
we mean an orthogonal $G$-representation, i.e., an inner product space $V$ 
equipped with a continuous $G$-action by linear isometries.
For every orthogonal space $Y$ and every $G$-representation $V$,
the value~$Y(V)$ inherits a $G$-action from~$V$ through the functoriality of $Y$. 
For a $G$-equivariant linear isometric embedding $\varphi:V\to W$
the induced map $Y(\varphi):Y(V)\to Y(W)$ is $G$-equivariant.

Now we discuss the equivariant homotopy set~$\pi_0^G(Y)$ of an orthogonal
space~$Y$; this is an unstable precursor of the 0-th equivariant stable
homotopy group of an orthogonal spectrum. 

\begin{defn} 
Let $G$ be a compact Lie group. 
A {\em $G$-universe} is an orthogonal $G$-representation $\Uc$ 
of countably infinite dimension with the following two properties:
\begin{itemize}
\item the representation $\Uc$ has non-zero $G$-fixed points,
\item if a finite dimensional representation $V$ embeds into $\Uc$,
then a countable infinite sum of copies of~$V$ also embeds into $\Uc$.
\end{itemize}
A $G$-universe is {\em complete} if every finite dimensional $G$-representation
embeds into it.
\end{defn}

A $G$-universe is characterized, up to equivariant isometry,
by the set of irreducible $G$-representations that embed into it.
A universe is complete if and only if 
every irreducible $G$-representation embeds into it.
In the following we fix, for every compact Lie group $G$,
a complete $G$-universe $\Uc_G$.
We let $s(\Uc_G)$ denote the poset, under inclusion,
of finite dimensional $G$-subrepresentations of $\Uc_G$.

We let $Y$ be an orthogonal space and define its
$G$-equivariant path components as
\begin{equation}\label{eq:define_pi_0^G_set}
   \pi_0^G(Y)  \ = \ \colim_{V\in s(\Uc_G)}\,  \pi_0\left( Y(V)^G\right)  \ .
\end{equation}
As the group varies, the homotopy sets $\pi_0^G(Y)$ 
have contravariant functoriality in $G$:
every continuous group homomorphism $\alpha:K\to G$
between compact Lie groups
induces a restriction map $\alpha^*:\pi_0^G(Y)\to\pi_0^K(Y)$,
as we shall now explain.
We denote by $\alpha^*$ the restriction functor from $G$-spaces to $K$-spaces 
(or from $G$-representations to $K$-representations), 
i.e., $\alpha^* Z$ (respectively $\alpha^*V$)
is the same topological space as~$Z$ 
(respectively the same inner product space as~$V$) endowed with $K$-action via
\[  k\cdot z \ = \ \alpha(k)\cdot z \ . \]
Given an orthogonal space $Y$,
we note that for every $G$-representation $V$, the $K$-spaces 
$\alpha^*(Y(V))$ and $Y(\alpha^*V)$ are equal (not just isomorphic).

The restriction $\alpha^*(\Uc_G)$ is a $K$-universe, 
but if $\alpha$ has a non-trivial kernel, then this $K$-universe is 
not complete. When $\alpha$ is injective,
then $\alpha^*(\Uc_G)$ is a complete $K$-universe, but typically
different from the chosen complete $K$-universe $\Uc_K$. 
To deal with this we explain how 
a $G$-fixed point $y\in Y(V)^G$,
for an arbitrary $G$-representation $V$,
gives rise to an unambiguously defined element $\td{y}$ 
in $\pi_0^G(Y)$.
The point here is that $V$ need not be a subrepresentation
of the chosen universe $\Uc_G$
and the resulting class does not depend on any additional choices.
To construct $\td{y}$ we choose a linear isometric $G$-embedding $j:V\to \Uc_G$
and look at the image $Y(j)(y)$ under the $G$-map
\[  Y(V)\ \xra{\ Y(j)\ }\ Y(j(V))\ . \]
Here we have used the letter $j$ to also denote the isometry
$j:V\to j(V)$ to the image of $V$; since $j(V)$ 
is a finite dimensional $G$-invariant subspace of $\Uc_G$, we obtain an element 
\[   \td{y} \ = \ [Y(j)(y)] \ \in \ \pi_0^G(Y)\ .  \]
It is crucial, although not particularly difficult,
that $\td{y}$ does not depend on the choice of embedding $j$:

\begin{prop}\label{prop-universal colimit spaces}
Let $Y$ be an orthogonal space, $G$ a compact Lie group,
$V$ a $G$-representation and $y\in Y(V)^G$ a $G$-fixed point.
\begin{enumerate}[\em (i)]
\item The class $\td{y}$ in $\pi_0^G(Y)$ is independent of the choice 
of linear isometric embedding $j:V\to\Uc_G$. 
\item For every $G$-equivariant linear isometric embedding 
$\varphi:V\to W$ the relation
\[ \td{Y(\varphi)(y)} \ = \ \td{y} \text{\qquad holds in\quad $\pi_0^G(Y)$.}  \]
\end{enumerate}
\end{prop}
\begin{proof}
(i) We let $j':V\to\Uc_G$ be another  $G$-equivariant linear isometric embedding. 
If the images $j(V)$ and $j'(V)$ are orthogonal, 
then~$H:V\times [0,1]\to j(V)\oplus j'(V)$ defined by
\[ H(v,t)\ =\ \sqrt{1-t^2}\ \cdot j(v)  \ + \ t\cdot j'(v) \]
is a homotopy from $j$ to $j'$ through $G$-equivariant linear isometric embeddings.
Thus
\[ t\ \longmapsto \ Y(H(-,t))(y) \]
is a path in $Y(j(V)\oplus j'(V))^G$ from $Y(j)(y)$ to $Y(j')(y)$, so
$[Y(j)(y)]=[Y(j')(y)]$ in $\pi_0^G(Y)$.
In general we can choose a third $G$-equivariant linear isometric embedding 
$l:V\to\Uc_G$ whose image is orthogonal 
to the images of~$j$ and~$j'$. Then $[Y(j)(y)]=[Y(l)(y)]=[Y(j')(y)]$
by the previous paragraph.

(ii) If $j:W\to\Uc_G$ is an equivariant linear isometric embedding,
then so is $j\varphi:V\to\Uc_G$.
Since we can use any equivariant isometric embedding to define 
the class $\td{y}$, we get
\[  \td{Y(\varphi)(y)} \ = \ [Y(j)(Y(\varphi)(y))] \ = \ [Y(j\varphi)(y)] 
\ = \ \td{y} \ . \qedhere \]
\end{proof}

We can now define the {\em restriction map} associated to
a continuous group homomorphism $\alpha:K\to G$ by
\[  \alpha^* \ : \ \pi^G_0(Y) \ \to \ \pi^K_0(Y)  \ ,\quad
[y]\ \longmapsto \ \td{y}\ . \]
This makes sense because every $G$-fixed point of $Y(V)$ is also 
a $K$-fixed point of $\alpha^*(Y(V))=Y(\alpha^* V)$.
For a second continuous group homomorphism $\beta:L\to K$ we have
\[ \beta^*\circ \alpha^* \ = \ (\alpha\beta)^* \ : \ \pi_0^G(Y) \ \to \ \pi_0^L(Y) \ .  \]
Since restriction along the identity homomorphism is the identity,
the collection of equivariant homotopy sets $\pi_0^G(Y)$
becomes a contravariant functor in the group variable.
A key fact is that inner automorphisms act trivially:

\begin{prop}\label{prop-inner automorphism} 
For every orthogonal space $Y$, every compact Lie group $G$ 
and every $g\in G$, the restriction map
$c_g^*:\pi_0^G(Y) \to \pi_0^G(Y)$ along the inner automorphism
\[  c_g \ : \ G \ \to \  G\ , \quad c_g(h)\ =\ g^{-1}h g  \]
is the identity of $\pi_0^G(Y)$.
\end{prop}
\begin{proof} 
We consider a finite dimensional $G$-subrepresentation $V$ 
of $\Uc_G$ and a $G$-fixed point $y\in Y(V)^G$ that represents an element in~$\pi_0^G(Y)$. 
Then the map $l_g:c_g^*(V)\to \Uc$ given by left multiplication by $g$
is a $G$-equivariant linear isometric embedding.
So 
\[ c_g^*[y]\ = \ [Y(l_g^V)(y)]\ = \ [g\cdot y]\ = \ [y] \ ,\]
by the very definition of the restriction map,
where $l_g^V:c_g^*(V)\to V$. The second equation is the definition
of the $G$-action on $Y(V)$ through the $G$-action on $V$.
The third equation is the hypothesis that $y$ is $G$-fixed.
\end{proof}

We denote by $\Rep$ the category whose objects are the compact Lie groups
and whose morphisms are conjugacy classes of continuous group homomorphisms.
We can summarize the discussion thus far by saying that for every
orthogonal space $Y$ the restriction maps
make the equivariant homotopy sets
$\{\pi_0^G(Y)\}$ into a contravariant functor
\[ \upi_0(Y) \ : \ \Rep \ \to \text{(sets)} \ . \]
In fact, the restriction maps along continuous homomorphisms
give {\em all} natural operations:
As we show in~\cite{schwede-global},
every natural transformation $\pi_0^G \to \pi_0^K$
of set valued functors on the category of orthogonal spaces is of the form
$\alpha^*$ for a unique conjugacy class of continuous group homomorphism $\alpha:K\to G$.

\medskip

If~$V$ is any inner product space, then the evaluation functor sending
an orthogonal space~$Y$ to~$Y(V)$ is represented by the hom functor~$\bL(V,-)$.
Consequently, if~$V$ is a $G$-representation, then the functor
\[  \spc \ \to \ \mathcal T \ , \quad Y \ \longmapsto \ Y(V)^G \]
that sends an orthogonal space $Y$ to the space of $G$-fixed points of $Y(V)$
is represented by an orthogonal space $\bL_{G,V}$,
the {\em free orthogonal space}
generated by $(G,V)$.
The value of~$\bL_{G,V}$ at an inner product space~$W$ is
\[ \bL_{G,V}(W) \ = \ \bL(V,W)/G \ , \]
the orbit space of the right $G$-action on $\bL(V,W)$ by
$(\varphi\cdot g)(v) = \varphi(g\cdot v)$.
Every $G$-fixed point $y\in Y(V)^G$ gives rise to a morphism
$\hat y:\bL_{G,V} \to Y$ of orthogonal spaces, defined at~$W$ as
\[   \hat y(W)\ : \  \bL(V,W) / G \ \to \ Y(W)\ ,\quad 
\varphi \cdot G \ \longmapsto\ Y(\varphi)(y)\ .\]
The morphism $\hat y$ is uniquely determined by the property
$\hat y(V)(\Id_V\cdot G)= y$ in $Y(V)^G$.

We calculate the 0-th equivariant homotopy sets of a free orthogonal space. 
The {\em tautological class}
\begin{equation}\label{eq:tautological_class}
 u_{G,V}\ \in \ \pi_0^G(\bL_{G,V})   
\end{equation}
is the path component of the $G$-fixed point
\[  \Id_V\cdot G \ \in \ (\bL(V,V)/G)^G = (\bL_{G,V}(V))^G\ , \]
the $G$-orbit of the identity of $V$.

\begin{theorem}\label{thm-pi_0 of L_G}
Let $K$ and $G$ be compact Lie groups and $V$ a faithful $G$-representation.
Then the map
\[  \Rep(K,G) \ \to \ \pi_0^K(\bL_{G,V})\ , \quad
[\alpha:K\to G] \ \longmapsto \ \alpha^*(u_{G,V})\]
is bijective. 
\end{theorem}
\begin{proof}
We construct the inverse explicitly. We consider any element
\[ [\varphi G]\ \in \ \pi_0^K(\bL_{G,V})\ ; \]
here $W\in s(\Uc_K)$, and $\varphi\in\bL(V,W)$ is such that the
orbit $\varphi G\in \bL(V,W)/G$ is $K$-fixed.
Thus $k\varphi G=\varphi G$ for every element $k\in K$.
Since $G$ acts faithfully on~$V$, there is a unique $\alpha(k)\in G$
with $k\varphi =\varphi \alpha(k)$, and this defines a 
continuous homomorphism $\alpha:K\to G$.
If we replace~$\varphi$ by~$\varphi g$ for some $g\in G$, then
$\alpha$ changes into its $g$-conjugate.
If we replace~$W$ by a larger $K$-representation in the poset~$s(\Uc_K)$,
then~$\alpha$ does not change.

Now we consider a path
\[ \omega \ : \ [0,1]\ \to \  (\bL(V,W)/G)^K \]
starting with~$\varphi G$. Since the projection
$\bL(V,W)\to\bL(V,W)/G$ is a locally trivial fiber bundle, we can choose
a continuous lift
\[  \tilde\omega \ : \ [0,1]\ \to \ \bL(V,W) \]
with $\tilde\omega(0)=\varphi$ and $\tilde\omega(t)G =\omega(t)$
for all $t\in[0,1]$.
Then each~$t$ determines a continuous homomorphism $\alpha_t:K\to G$
by $k\tilde\omega(t) =\tilde\omega(t)\alpha_t(k)$, and the assignment
\[ [0,1]\ \to \ \Hom(K,G)\ , \quad t \longmapsto \alpha_t \]
to the space of continuous group homomorphisms 
(with the topology of uniform convergence) is itself continuous.
But that means that~$\alpha_0$ and~$\alpha_1$ are conjugate by
an element of~$G$, compare~\cite[VIII, Lemma 38.1]{conner-floyd}.
In particular, the conjugacy class of~$\alpha$ only depends 
on the path component of $\varphi G$ in the space $(\bL(V,W)/G)^K$. 
Altogether this shows that the map
\[ \pi_0^K(\bL_{G,V})\ \to \ \Rep(K,G) \ , \quad [\varphi G]\ \longmapsto \ [\alpha] \]
is well-defined.
It is straightforward from the definitions that this map is inverse 
to evaluation at~$u_{G,V}$.
\end{proof}

We end this section by discussing certain orthogonal spaces
that are closely related to the symmetric product filtration.

\begin{construction}
For an inner product space $V$ we set
\[ S(V,n)\ = \ \left\{ (v_1,\dots,v_n)\in V^n \ : \ \sum_{i=1}^n v_i=0\ , \ \sum_{i=1}^n |v_i|^2 = 1 
\right\} \ .\]
In other words, $S(V,n)$ is the unit sphere in the kernel of summation map
from $V^n$ to~$V$.
The symmetric group~$\Sigma_n$ acts from the right on $S(V,n)$ 
by permuting the coordinates, i.e.,
\[ (v_1,\dots,v_n)\cdot\sigma \ = \ (v_{\sigma(1)},\dots,v_{\sigma(n)}) \ .  \]
We define
\[ (B_{\gl}\Fc_n)(V) \ = \ S(V,n) / \Sigma_n \ ,\]
the orbit space of the $\Sigma_n$-action.
A linear isometric embedding $\varphi:V\to W$ induces the map
\[ (B_{\gl}\Fc_n)(\varphi) \ = \ S(\varphi,n) / \Sigma_n  \ , \quad
(v_1,\dots,v_n)\Sigma_n\ \longmapsto \ (\varphi(v_1),\dots,\varphi(v_n))\Sigma_n\ .\]
We call~$B_{\gl}\Fc_n$ the {\em global classifying space} of the family $\Fc_n$
of non-transitive subgroups of the symmetric group~$\Sigma_n$. 
Proposition~\ref{prop-universal space} below justifies this terminology.
\end{construction}

 \begin{rk}
   The {\em reduced natural $\Sigma_n$-representation} 
   (also called the {\em standard $\Sigma_n$-representation})
   is the vector space 
   \[ \nu_n\ = \ \{(x_1,\dots,x_n)\in\mR^n \ : \ x_1+\ldots + x_n = 0 \} \]
   with the standard scalar product
   and left~$\Sigma_n$-action by permutation of coordinates:
   \[ \sigma\cdot(x_1,\dots,x_n) \ = \ (x_{\sigma^{-1}(1)},\dots,x_{\sigma^{-1}(n)}) \ .\]
   In the proof of the following proposition we exploit that 
   for every inner product space~$V$ (possibly infinite dimensional),
   the kernel of the summation map $V^n\to V$
   is isometrically and $(O(V)\times\Sigma_n)$-equivariantly isomorphic
   to $V\tensor\nu_n$.
   Hence $S(V,n)$ is $(O(V)\times\Sigma_n)$-equivariantly homeomorphic
   to~$S(V\tensor\nu_n)$.
 \end{rk}

We show now that for every compact Lie group $K$
the $K$-space $(B_{\gl}\Fc_n)(\Uc_K)=S(\Uc_K,n)/\Sigma_n$
is a certain classifying space,
thereby justifying the term `global classifying space' for~$B_{\gl}\Fc_n$.
We denote by $\Fc_n(K)$ the family 
of those closed subgroups $\Gamma$ of $K\times\Sigma_n$
whose trace $H=\{\sigma\in \Sigma_n\ |\ (1,\sigma)\in \Gamma\}$
is a non-transitive subgroup of~$\Sigma_n$.
For the purpose of the next proposition we combine the left~$K$-action
and the right $\Sigma_n$-action on $S(\Uc_K,n)$ into a left action
of~$K\times\Sigma$ by
\begin{equation}\label{eq:turn_into_left}
 (k,\sigma)\cdot (v_1,\dots,v_n) \ = \ 
(k\cdot  v_{\sigma^{-1}(1)},\dots,k\cdot v_{\sigma^{-1}(n)}) \ .  
\end{equation}

\begin{prop}\label{prop-universal space} 
Let $K$ be a compact Lie group and $n\geq 2$.
Then the $(K\times\Sigma_n)$-space $S(\Uc_K,n)$
is a universal space for the family $\Fc_n(K)$
of subgroups of $K\times\Sigma_n$.
\end{prop}
\begin{proof}
We let~$\Gamma$ be a closed subgroup of $K\times\Sigma_n$.
If the trace $H=\{\sigma\in\Sigma_n\ |\ (1,\sigma)\in \Gamma\}$
is a transitive subgroup of $\Sigma_n$, then all
$H$-fixed points of $S(\Uc_K,n)$ are diagonal, i.e.,
of the form $(v,\dots,v)$ for some $v\in \Uc_K$. Since the components
must add up to~0, this forces $v=0$, which cannot happen for tuples
in the unit sphere. So if $H$ is a transitive subgroup,
then $S(\Uc_K,n)$ has no $H$-fixed points, hence no $\Gamma$-fixed points either.

Now we suppose that the trace~$H$ is not transitive.
We view the subgroup $\Gamma\leq K\times\Sigma_n$ as a generalized graph:
we denote by~$L$ the image of~$\Gamma$
under the projection $K\times\Sigma_n\to K$ and define a group homomorphism
$\beta:L\to W_{\Sigma_n}H$ to the Weyl group of the trace $H$ by
\[ \beta(l) \ = \ 
\{ \sigma\in \Sigma_n\ |\ (l,\sigma)\in \Gamma\} \ \in \ W_{\Sigma_n} H\ .\]
We can recover~$\Gamma$ as the graph of~$\beta$, i.e.,
\[ \Gamma \ = \ {\bigcup}_{l\in L}  \{l\}\times \beta(l) \ .\]
We let~$L$ act on $(\nu_n)^H$ by restriction along~$\beta$;
then $\beta^*((\nu_n)^H)$ is a non-zero $L$-representation
because $H$ is non-transitive.
Since $\Uc_K$ is a complete $K$-universe, the underlying $L$-universe 
is also complete, hence so is the $L$-universe 
$\Uc_K\tensor \beta^*((\nu_n)^H)$. So
\[   (S(\Uc_K\tensor\nu_n))^\Gamma \ = \ S((\Uc_K\tensor\beta^*((\nu_n)^H))^L)  \]
is an infinite dimensional  unit sphere, hence contractible.
\end{proof}

We define a specific class in the equivariant homotopy set
$\pi_0^{\Sigma_n}(B_{\gl}\Fc_n)$. We set 
\[ b\ =\ (1/n,\dots,1/n)\ \in\ \mR^n \]
and let $e_i$ be the $i$-th vector of the canonical basis of~$\mR^n$. 
Then $b-e_i$ lies in the reduced natural~$\Sigma_n$-representation~$\nu_n$.
Because~$|b-e_i|^2=\frac{n-1}{n}$, the vector
\[ D_n \ = \  \frac{1}{\sqrt{n-1}}(b- e_1,\dots,b-e_n) \ \in \ (\nu_n)^n \]
lies in the unit sphere $S(\nu_n,n)$.
The~$\Sigma_n$-orbit of $D_n$ 
(with respect to the right action permuting the `outer' coordinates)
is $\Sigma_n$-fixed (with respect to the left action permuting the `inner' coordinates),
i.e.,
\[ D_n\cdot\Sigma_n\ \in \ 
( S(\nu_n,n)/\Sigma_n)^{\Sigma_n} \ = \ 
((B_{\gl}\Fc_n)(\nu_n))^{\Sigma_n} \ .\]
We denote by 
\begin{equation}\label{eq:define_u_n}
 u_n\ = \ \td{D_n\cdot \Sigma_n} \ \in\ \pi_0^{\Sigma_n} (B_{\gl}\Fc_n)  
\end{equation}
the class represented by this $\Sigma_n$-fixed point.
The next theorem says that the class $u_n$ generates
the homotopy set Rep-functor $\upi_0 (B_{\gl}\Fc_n)$.

\begin{theorem}\label{thm-u_n generates T_n} 
For every $n\geq 2$, every compact Lie group $K$ and every element 
$x$ of $\pi_0^K(B_{\gl}\Fc_n)$ there is a continuous group 
homomorphism $\alpha:K\to\Sigma_n$ such that $\alpha^*(u_n)=x$.  
\end{theorem}
\begin{proof}
An element of $\pi_0^K(B_{\gl}\Fc_n)$ is represented by 
a $K$-representation~$V$ in~$s(\Uc_K)$ and a $K$-fixed  $\Sigma_n$-orbit 
\[ v\cdot \Sigma_n\ \in \   S(V,n)/\Sigma_n \ = \ (B_{\gl}\Fc_n)(V)\ .  \]
We let~$H$ denote the $\Sigma_n$-stabilizer of~$v$,
a non-transitive subgroup of $\Sigma_n$.
We define a continuous homomorphism 
$\beta:K\to W_{\Sigma_n}H$ to the Weyl group of $H$ by 
\[ \beta(k) \ = \ \{ \sigma\in \Sigma_n\ |\ k v = v \sigma\} \ .\]
As the $\Sigma_n$-stabilizer of a point in $V^n$, the group~$H$ is
a Young subgroup of~$\Sigma_n$, i.e., the product of the symmetric groups
of all the orbits of the tautological $H$-action on $\{1,\dots,n\}$.
Thus the projection $q:N_{\Sigma_n} H\to W_{\Sigma_n} H$ has a multiplicative section
$s:W_{\Sigma_n} H\to N_{\Sigma_n} H$.
We define $\alpha:K\to\Sigma_n$ as the composite homomorphism
\[ K \ \xra{\ \beta\ }\ W_{\Sigma_n}H  \ \xra{\ s\ }\  
N_{\Sigma_n} H \ \xra{\text{incl}}\  \Sigma_n  \]
and claim that
\[  \alpha^*(u_n) \ = \ [v\cdot\Sigma_n] \text{\qquad in \quad} 
\pi_0^K(B_{\gl}\Fc_n)\ .   \]

We turn $S(V,n)$ into a left~$(K\times\Sigma_n)$-space as in~\eqref{eq:turn_into_left}
and let~$\Gamma\leq K\times\Sigma_n$ denote the graph of~$\alpha$. 
Since $\alpha(k)\in\beta(k)$ for every $k\in K$, the vector~$v$ is fixed
by~$\Gamma$.
Increasing the $K$-representation $V$ does not change the
stabilizer group of the vector $v$ nor the class represented by the orbit
$v\cdot\Sigma_n$ in $\pi_0^K(B_{\gl}\Fc_n)$; 
we can thus assume without loss of generality
that there is a $K$-equivariant linear isometric embedding 
$\varphi:\alpha^*(\nu_n)\to V$.
As the $K$-representations $V$ exhaust a complete $K$-universe, the 
$(K\times\Sigma_n)$-spaces $S(V,n)$ approximate a universal
space for the family $\Fc_n(K)$, by Proposition~\ref{prop-universal space}.
The graph~$\Gamma$ of~$\alpha$ belongs to $\Fc_n(K)$,
so after increasing the $K$-representation $V$, if necessary,
we can assume that the dimension of the fixed point sphere
$S(V,n)^\Gamma$ is at least~1, so that this fixed point space is path connected.
The class $\alpha^*(u_n)$ is represented by the $\Sigma_n$-orbit
of the point
\[ S(\varphi,n)(D_n) \ \in \ S(V,n)^\Gamma \]
and the original class in $\pi_0^K(B_{\gl}\Fc_n)$
is represented by the vector~$v$.
Any path between~$S(\varphi,n)(D_n)$ and~$v$ 
in the fixed point space $S(V,n)^\Gamma$
projects to a path of $K$-fixed points between the orbits
\[  S(\varphi,n)(D_n)\cdot \Sigma_n \ , \quad v\cdot \Sigma_n \quad \in \ 
(S(V,n) / \Sigma_n)^K \ = \ \left((B_{\gl}\Fc_n)(V)\right)^K\ .  \]
This proves that $\alpha^*(u_n) =[v\cdot\Sigma_n]$, and it finishes the proof.
\end{proof}

\begin{rk}\label{rk-BF_2 is BSigma_2}
The orthogonal space~$B_{\gl}\Fc_2$ is isomorphic
to the free orthogonal space generated by $(\Sigma_2,\sigma)$,
where $\sigma$ is the 1-dimensional sign representation of $\Sigma_2$ on~$\mR$.
An isomorphism of orthogonal spaces 
\[ \bL_{\Sigma_2,\sigma}\ \iso \ B_{\gl} \Fc_2 \]
is induced at an inner product space $V$ by the $\Sigma_2$-equivariant 
natural homeomorphism
\[ \bL(\sigma, V)\ \iso \  S(V,2) \ , \quad  (\varphi:\sigma\to V)
\ \longmapsto \  \left( \varphi(1/\sqrt{2}),-\varphi(1/\sqrt{2}) \right)\ .\]
This isomorphism sends 
the tautological class~$u_{\Sigma_2,\sigma}$  (see~\eqref{eq:tautological_class})
in~$\pi_0^{\Sigma_2}(\bL_{\Sigma_2,\sigma})$ to the class $u_2\in \pi_0^{\Sigma_2}(B_{\gl}\Fc_2)$.
So by Theorem~\ref{thm-pi_0 of L_G}
every element of $\pi_0^K (B_{\gl}\Fc_2)$
is of the form~$\alpha^*(u_2)$ for a {\em unique} conjugacy classes
of continuous group homomorphism $\alpha:K\to\Sigma_2$.
For $n\geq 3$, however, $\alpha$ is typically not
unique up to conjugacy, and $\upi_0 (B_{\gl}\Fc_n)$ is {\em not}
a representable Rep-functor.
\end{rk}

\section{Orthogonal spectra}

In this section we recall orthogonal spectra, 
the objects that represent {\em stable} global homotopy types.
Orthogonal spectra are used, at least implicitly, in~\cite{may-quinn-ray},
and the term `orthogonal spectrum' was introduced in~\cite{mmss}, where a
non-equivariant stable model structure for orthogonal spectra
was constructed.
Before giving the formal definition of orthogonal spectra we try to motivate it.
An orthogonal space $Y$ assigns values 
to all finite dimensional inner product spaces.
Informally speaking, the global homotopy type is encoded in
the $G$-spaces obtained as the `homotopy colimit of $Y(V)$
over all $G$-representations~$V$'.
So besides the values~$Y(V)$, we use the $O(V)$-action 
(which is turned into a $G$-action when~$G$ acts on~$V$)
and the information about inclusions of inner product spaces.
All this information is conveniently encoded as a continuous functor 
from the category~$\bL$.

An orthogonal spectrum~$X$ is a stable analog of this: it assigns
a based space~$X(V)$ to every inner product space, and it keeps
track of an $O(V)$-action on~$X(V)$ (to get $G$-homotopy types when $G$ acts on~$V$)
and of a way to stabilize by suspensions.
When doing this in a coordinate free way, the stabilization data
assigns to a linear isometric embedding~$\varphi:V\to W$ a 
continuous based map
\[ \varphi_\star \ : \ X(V)\sm S^{\varphi^\perp}\ \to \ X(W) \]
that `varies continuously with~$\varphi$'.
To make the continuous dependence rigorous one exploits 
that the orthogonal complements $\varphi^\perp$ vary in a locally trivial way, 
i.e., they are the fibers of an `orthogonal complement' vector bundle
over the space of~$\bL(V,W)$ of linear isometric embeddings.
All the structure maps~$\varphi_\star$ together define a map 
on the smash product of~$X(V)$ with the Thom space of this complement bundle,
and the continuity in~$\varphi$ is formalized by
requiring continuity of that map.
The Thom spaces together form the morphism spaces of a based
topological category, and the data of an orthogonal spectrum 
can conveniently be packaged as a continuous based
functor on this category.

\begin{construction}
We let $V$ and $W$ be inner product spaces. 
The `orthogonal complement' vector bundle 
over the space~$\bL(V,W)$
is the subbundle of the trivial vector bundle $W\times\bL(V,W)$
with total space
\[ \xi(V,W) \ = \ 
\{\, (w,\varphi) \in W\times\bL(V,W) \ | \ \text{$\td{w,\varphi(v)}=0$ for all $v\in V$}\,\} \ .\]
The fiber over $\varphi:V\to W$ is the orthogonal complement of the image of $\varphi$.

We let $\bO(V,W)$ be the one-point compactification of the total space of $\xi(V,W)$;
since the base space~$\bL(V,W)$ is compact, $\bO(V,W)$ is also the Thom space 
of the bundle~$\xi(V,W)$.
Up to non-canonical homeomorphism, we can describe the space $\bO(V,W)$ 
differently as follows: if $\dim V=n$ and $\dim W=n+m$,
then $\bL(V,W)$ is homeomorphic to the homogeneous space $O(n+m)/O(m)$
and $\bO(V,W)$ is homeomorphic to $O(n+m)_+\sm_{O(m)}S^m$.

The spaces $\bO(V,W)$ are the morphism spaces of a based topological category.
Given a third inner product space $U$, the bundle map 
\[ \xi(V,W) \times \xi(U,V) \ \to \ \xi(U,W) \ , \quad
((w,\varphi),\,(v,\psi)) \ \longmapsto \ (w+\varphi(v),\,\varphi\psi)\]
covers the composition in~$\bL$.
Passage to one-point compactification gives a based map
\[ \circ \ : \ \bO(V,W) \sm \bO(U,V) \ \to \ \bO(U,W) \]
which is the composition in the category $\bO$. 
The identity of $V$ is $(0,\Id_V)$ in $\bO(V,V)$.
\end{construction}

\begin{defn}\label{def-orthogonal spectrum}
An {\em orthogonal spectrum}
is a based continuous functor from~$\bO$ to the category of based spaces.
A {\em morphism} is a natural transformation of functors.
We denote by $\spec$ the category of orthogonal spectra.
\end{defn}

We denote by $S^V$ the one-point compactification of an inner product space~$V$,
with basepoint at infinity. 
If~$X$ is an orthogonal spectrum and~$V$ and~$W$ inner product spaces,
we define the {\em structure map}
\[  \sigma_{V,W} \ : \ X(V)\sm S^W \ \to\  X(V\oplus W)  \]
as the composite
\[  X(V)\sm S^W\ \xra{X(V)\sm ((0,-),i_V)} \  X(V)\sm \bO(V,V\oplus W) 
\ \xra{\ X\ }\  X(V\oplus W)  \]
where $i_V:V\to V\oplus W$ is the inclusion of the first summand.
If a compact Lie group $G$ acts on~$V$ and~$W$ by linear isometries, then
$X(V)$ becomes a based $G$-space by 
restriction of the action of $\bO(V,V)=O(V)_+$,
and the structure map~$\sigma_{V,W}$ is $G$-equivariant.

\begin{rk}\label{rk-pointset properties} 
Given an orthogonal spectrum $X$ and a compact Lie group $G$,
the collection of $G$-spaces~$X(V)$ and the structure maps
$\sigma_{V,W}$ form an {\em orthogonal $G$-spectrum}
in the sense of~\cite{mandell-may} that we denote by~$X_G$.
However, only very special orthogonal $G$-spectra arise in this way from
an orthogonal spectrum. More precisely, an orthogonal $G$-spectrum $Y$
is isomorphic to~$X_G$ for some orthogonal spectrum $X$ if and only if 
for every {\em trivial} $G$-representation $V$, the $G$-action on~$Y(V)$ is trivial. 
An orthogonal $G$-spectrum that does not satisfy
this condition is the equivariant suspension spectrum of a 
based $G$-space with non-trivial $G$-action.
In Remark~\ref{rk-homotopy properties}
below we isolate some conditions on the Mackey functor
homotopy groups of an orthogonal $G$-spectrum that hold 
for all $G$-spectra of the special form $X_G$.
\end{rk}

As we just explained,
an orthogonal spectrum $X$ has an underlying orthogonal $G$-spectrum
for every compact Lie group $G$.
As such, it has equivariant stable homotopy groups,
whose definition we now recall.
As before, $s(\Uc_G)$ denotes the poset, under inclusion, 
of finite dimensional $G$-subrepresentations 
of the complete~$G$-universe $\Uc_G$.
For $k\geq 0$ we consider the functor from~$s(\Uc_G)$ to sets 
that sends $V\in s(\Uc_G)$ to 
\[  [S^{k+V}, X(V)]^G \ ,\]
the set of $G$-equivariant homotopy classes of based $G$-maps
from $S^{k+V}$ to $X(V)$ (where $k+V$ is short hand for $\mR^k\oplus V$ 
with trivial $G$-action on~$\mR^k$).
The map induced by an inclusion $V\subseteq W$ in $s(\Uc_G)$  sends the homotopy class
of~$f:S^{k+V}\to X(V)$ to the class of the composite
\[ S^{k+W} \iso \ S^{k+ V}\sm S^{W-V} \ \xra{\ f\sm S^{W-V}} \ 
X(V)\sm S^{W-V} \ \xra{\sigma_{V,W-V}}\ X(V\oplus (W-V))\ = \ X(W) \ ,  \]
where $W-V$ is the orthogonal complement of~$V$ in~$W$.
The {\em $k$-th equivariant homotopy group} $\pi_0^G(X)$ 
is then defined as
\begin{equation}\label{eq:defined_pi_0^G}
 \pi_k^G(X) \ = \ \colim_{V\in s(\Uc_G)}\, [S^{k+V}, X(V)]^G  \ ,  
\end{equation}
the colimit of this functor over the poset~$s(\Uc_G)$.
For $k< 0$, the definition is essentially the same, but we take a colimit
over~$s(\Uc_G)$ of the sets $[S^V, X(\mR^{-k}\oplus V)]^G$.
When the fixed points $V^G$ have dimension at least~2, then
$[S^V,X(V)]^G$ comes with a commutative group structure, and the maps 
out of it are homomorphisms. 
The $G$-subrepresentations $V$ with $\dim(V^G)\geq 2$
are cofinal in the poset $s(\Uc_G)$, so the abelian group structures on $[S^V,X(V)]^G $
for  $\dim(V^G)\geq 2$ assemble into a well-defined and
natural abelian group structure on the colimit $\pi_0^G(X)$.
The argument for~$\pi_k^G(X)$ is similar.

\begin{defn}\label{def-global equivalence}
A morphism $f:X\to Y$ of orthogonal spectra 
is a {\em global equivalence}
if the induced map $\pi_k^G(f):\pi_k^G(X) \to \pi_k^G(Y)$ is an isomorphism
for all compact Lie groups $G$ and all integers~$k$.
\end{defn}

The {\em global stable homotopy category}
is the category obtained from the category of orthogonal spectra 
by formally inverting the global equivalences. 
The global equivalences are the weak equivalences of the
{\em global model structure} on the category of orthogonal spectra,
see~\cite{schwede-global}.
So the methods of homotopical algebra are available for studying global
equivalences and the associated global homotopy category.

\medskip

Now we set up the formalism of {\em global functors},
the natural (in fact, the tautological) home of the collection of equivariant
homotopy groups of an orthogonal spectrum.
In this language we then describe
the equivariant homotopy groups  $\pi_0^G(Sp^n)$ of the symmetric product
spectrum~$Sp^n$ as a whole, i.e., when 
the compact Lie group $G$ is varying:
the global functor $\upi_0( Sp^n )$ is the quotient of the Burnside ring global functor 
by a single basic relation.

\begin{defn}[Global Burnside category]
The {\em global Burnside category} $\bA$
has  all compact Lie groups as objects; 
the morphisms from a group $G$ to $K$ are defined as
\[ \bA(G,K) \ = \ \text{Nat}(\pi_0^G,\pi_0^K) \ ,\]
the set of natural transformations of functors, from orthogonal spectra to sets,
between the equivariant homotopy group functors~$\pi_0^G$ and $\pi_0^K$. 
Composition in~$\bA$ is composition of natural transformations.  
\end{defn}

It is not a priori clear that the natural transformations 
from~$\pi_0^G$ to~$\pi_0^K$ form a set (as opposed to a proper class),
but this follows from the representability result 
in Proposition~\ref{prop-B_gl represents} below. 
The functor $\pi_0^K$ is abelian group valued,
so the set $\bA(G,K)$ is an abelian group under objectwise addition of transformations.
Composition is additive in each variable, so $\bA(G,K)$ is a pre-additive category.

The Burnside category $\bA$ is skeletally small: isomorphic compact Lie groups
are also isomorphic in the category $\bA$, and every compact Lie
group is isomorphic to a closed subgroup of an orthogonal group~$O(n)$.

\begin{defn}\label{def-global functor} 
A {\em global functor} is an additive functor 
from the global Burnside category $\bA$ 
to the category of abelian groups.
A morphism of global functors is a natural transformation.
\end{defn}

As a category of additive functors out of a skeletally small 
pre-additive category, the category of global functors is abelian
with enough injectives and projectives.
The global Burnside category $\bA$ is designed so that
the collection of equivariant homotopy groups of an orthogonal spectrum 
is tautologically a global functor.
Explicitly, the global homotopy group functor $\upi_0(X)$
of an orthogonal spectrum $X$ is defined on objects by
\[ \upi_0(X)(G)\ = \ \pi_0^G(X) \]
and on morphisms by evaluating natural transformations at $X$.

It is less obvious that conversely every global functor is isomorphic
to the homotopy group global functor~$\upi_0(X)$ of some orthogonal spectrum $X$.
We show this in~\cite{schwede-global} by constructing
Eilenberg-Mac\,Lane spectra from global functors.
In fact, the full subcategories of globally connective 
respectively globally coconnective orthogonal spectra define a non-degenerate 
t-structure on the triangulated global stable homotopy category,
and the heart of this t-structure is (equivalent to) 
the abelian category of global functors.

The abstract definition of the global Burnside category $\bA$ is convenient for
formal considerations and for defining the global functor $\upi_0(X)$ 
associated to an orthogonal spectrum~$X$, 
but to facilitate calculations 
we should describe the groups $\bA(G,K)$ more explicitly.
As we shall explain, the operations between the equivariant homotopy groups
come from two different sources:  restriction maps
along continuous group homomorphisms
and transfer maps along inclusions of closed subgroups.
A quick way to define the restriction maps, 
and to deduce some of their properties, 
is to interpret~$\pi_0^G(X)$ as the $G$-equivariant homotopy set, 
as defined in~\eqref{eq:define_pi_0^G_set}, of a certain orthogonal space.

\begin{construction}\label{con-Omega^bullet}
We recall the functor
\[ \Omega^\bullet \ : \ \spec \ \to \ \spc \]
from orthogonal spectra to orthogonal spaces.
Given an orthogonal spectrum $X$, the value of 
$\Omega^\bullet X$ at an inner product space $V$ is
\[ (\Omega^\bullet X)(V)\ = \ \map(S^V,X(V))\ . \]
If $\varphi:V\to W$ is a linear isometric embedding, the induced map
\[ \varphi_* \ : \ (\Omega^\bullet X)(V)\ = \ \map(S^V,X(V)) \ \to \ 
 \map(S^W,X(W))\ = \ (\Omega^\bullet X)(W) \]
is by `conjugation and extension by the identity'.
In more detail: given a continuous based map $f:S^V\to X(V)$ 
we define $\varphi_*(f):S^W\to X(W)$ as the composite
\[ S^W \iso \ S^V\sm S^{\varphi^\perp}
\ \xra{ f\sm S^{\varphi^\perp}} \ 
X(V)\sm S^{\varphi^\perp} \ \xra{\sigma_{V,\varphi^\perp}}
X(V\oplus \varphi^\perp) \ \iso\ X(W) \ , \]
where each of the two unnamed homeomorphisms uses $\varphi$ to identify 
$V\oplus \varphi^\perp$ with $W$.
In particular, the orthogonal group $O(V)$ acts on 
$(\Omega^\bullet X)(V)$ by conjugation.
The assignment $(\varphi,f)\mapsto \varphi_*(f)$ is continuous in both variables
and functorial in~$\varphi$.
In other words, we have defined an orthogonal space $\Omega^\bullet X$.

The functor $\Omega^\bullet$ has a left adjoint, defined as follows.
To every orthogonal space $Y$ we can associate an unreduced
{\em suspension spectrum}
 $\Sigma^\infty_+Y$ whose value on an inner product space is given by
\[ (\Sigma^\infty_+Y)(V)\ = \ Y(V)_+\sm S^V\ ; \]
the structure map
\[\bO(V,W)\sm Y(V)_+\sm S^V \ \to \  Y(W)_+\sm S^W   \]
is given by
\[  (w,\varphi)\sm y\sm v \ \longmapsto \ Y(\varphi)(y)\sm (w+\varphi(v))\ .    \]
If $Y$ is the constant orthogonal space with value~$A$,
then $\Sigma^\infty_+Y$ specializes to the usual suspension spectrum
of~$A$ with a disjoint basepoint added.
\end{construction}

If $G$ acts on $V$ by linear isometries, then the 
$G$-fixed subspace of $(\Omega^\bullet X)(V)$ is the space
of $G$-equivariant based maps from $S^V$ to $X(V)$.
The path components of this space are precisely the equivariant homotopy
classes of based $G$-maps, i.e.,
\[ \pi_0\left(((\Omega^\bullet X)(V))^G\right)\ = \ 
\pi_0\left( \map^G(S^V,X(V))\right) \ = \ [S^V,X(V)]^G\ . \]
Passing to the colimit over the poset $s(\Uc_G)$ gives
\[ \pi_0^G(\Omega^\bullet X)\ = \ \pi_0^G(X)\ , \]
i.e., the $G$-equivariant homotopy group of the orthogonal spectrum $X$
equals the $G$-equivariant homotopy set 
(as previously defined in~\eqref{eq:define_pi_0^G_set})
of the orthogonal space $\Omega^\bullet X$.
So by specializing the restriction maps for orthogonal spaces
we obtain restriction maps
\[ \alpha^*\ : \   \pi_0^G(X)\ \to \ \pi_0^K(X)\]
for every continuous group homomorphism $\alpha:K\to G$.
These restriction maps are again contravariantly functorial
and depend only on the conjugacy class of~$\alpha$
(by Proposition~\ref{prop-inner automorphism}).
Moreover,~$\alpha^*$ is additive, i.e., a group homomorphism.

The transfer maps  
\[ \tr_H^G\ :\ \pi_0^H(X)\ \to\ \pi_0^G(X) \]
are the classical ones that arise from the orthogonal
$G$-spectrum underlying $X$; they are defined whenever~$H$ is a closed
subgroup of~$G$ and constructed by an equivariant
Thom-Pontryagin construction~\cite[Sec.\,IX.3]{may-alaska}, \cite{nishida-transfer}.
Transfers are additive and natural for homomorphisms of orthogonal spectra;
since we only consider degree~0 transfers (as opposed to more general
`dimension shifting transfers'), the transfer~$\tr_H^G$
is trivial whenever~$H$ has infinite index in its normalizer in~$G$.

As we shall now explain, 
the suspension spectrum functor `freely builds in' the extra structure
that is available at the level of $\pi_0^G$ for orthogonal {\em spectra} 
(as opposed to orthogonal {\em spaces}), 
namely the abelian group structure and transfers.
We let $Y$ be an orthogonal space and $G$ a compact Lie group.
We define a stabilization map
\begin{equation}  \label{eq:sigma_map}
 \sigma^G\ : \ \pi_0^G(Y) \ \to \ \pi_0^G(\Sigma^\infty_+ Y)  
\end{equation}
as the effect of the adjunction unit $Y\to\Omega^\bullet(\Sigma^\infty_+ Y)$
on the  $G$-equivariant homotopy set $\pi_0^G$, using the identification
$\pi_0^G(\Omega^\bullet(\Sigma^\infty_+ Y))=\pi_0^G(\Sigma^\infty_+ Y)$.
More explicitly: if $V$ is a finite dimensional $G$-subrepresentation
of the complete $G$-universe $\Uc_G$ and $y\in Y(V)^G$ a $G$-fixed point, then
$\sigma^G[y]$ is represented by the $G$-map
\[ S^V\ \xra{y\sm -} \ Y(V)_+\sm S^V \ = \ (\Sigma^\infty_+ Y)(V) \ .\]
The stabilization maps~\eqref{eq:sigma_map} commute with restriction,
since they arise from a morphism of orthogonal spaces.

For a closed subgroup $L$ of a compact Lie group $K$, the normalizer $N_K L$ acts
on $L$ by conjugation, and hence on~$\pi_0^L(Y)$ by restriction 
along the conjugation maps. Restriction along an inner automorphism
is the identity, so the action of $N_K L$ factors over an
action of the Weyl group $W_K L=N_K L/L$ on $\pi_0^L(Y)$.
After passing to the stable classes along the map
$\sigma^L:\pi_0^L(Y)\to\pi_0^L(\Sigma^\infty_+ Y)$, we can then transfer from $L$ to $K$.
For an element $k\in N_K L$ and a class $x\in\pi_0^L(Y)$ we have
\[ \tr_L^K(\sigma^L(c_k^*(x))) \ = \ \tr_L^K(c_k^*(\sigma^L(x))) \ = \
c_k^*(\tr_L^K(\sigma^L(x))) \ = \ \tr_L^K(\sigma^L(x)) \]
because transfer commutes with restriction along the conjugation maps
\[ c_k\ : \ L \ \to \ L \text{\qquad respectively\qquad} c_k \ : \ K \ \to \ K  \]
defined by $c_k(h)=k^{-1}h k$.
So transferring from $L$ to $K$ in the global functor $\upi_0(\Sigma^\infty_+ Y)$
equalizes the action of the Weyl group $W_K L$ on $\pi_0^L(Y)$.

\begin{prop}\label{prop-pi_0 of Sigma^infty} 
Let $Y$ be an orthogonal space. Then for every compact Lie group $K$ 
the equivariant homotopy group $\pi_0^K(\Sigma^\infty_+ Y)$
of the suspension spectrum of $Y$ is a free abelian group with a basis 
given by the elements
\[ \tr_L^K(\sigma^L(x)) \]
where $L$ runs through all conjugacy classes of closed subgroups of $K$
with finite Weyl group and $x$ runs through a set of representatives 
of the $W_K L$-orbits of the set $\pi_0^L(Y)$.
\end{prop}
\begin{proof}
We consider the functor on the product poset $s(\Uc_K)^2$ sending
$(V,U)$ to the set $[S^V, Y(U)_+\sm S^V]^K$.
The diagonal is cofinal in $s(\Uc_K)^2$, thus the induced map 
\[ \pi_0^K(\Sigma^\infty_+ Y)\ = \ \colim_{V\in s(\Uc_G)} 
 [S^V, Y(V)_+\sm S^V]^K \ \to \ \colim_{(V,U)\in s(\Uc_K)^2}  [S^V, Y(U)_+\sm S^V]^K \]
is an isomorphism.
The target can be calculated in two steps, so the group we are after
is isomorphic to
\[ \colim_{U\in s(\Uc_K)} \left( \colim_{V\in s(\Uc_K)}  [S^V, Y(U)_+\sm S^V]^K \right)
\ = \ \colim_{U\in s(\Uc_K)}  \pi_0^K\left(\Sigma^\infty_+ Y(U) \right)\ .\]
We may thus show that the latter group is free abelian with the specified basis.

The rest of the argument is well known, and a version of it
can be found in~\cite[V Cor.\,9.3]{lms}.
The tom\,Dieck splitting~\cite[Satz~2]{tomDieck-OrbittypenII} provides an isomorphism
\[ \bigoplus_{(L)} \pi_0^{W L}( \Sigma^\infty_+( E W L \times Y(U)^L) ) \ \iso \
 \pi_0^K(\Sigma^\infty_+ Y(U)) \ ,\]
where the sum is indexed over all conjugacy classes of closed subgroups $L$ 
and $W L=W_K L$ is the Weyl group of~$L$  in~$K$.
By~\cite[Sec.\,4]{tomDieck-OrbittypenII} the group 
$\pi_0^{W L}( \Sigma^\infty_+( E W L \times Y(U)^L))$ 
vanishes if the Weyl group $W L$ is
infinite; so only the summands with finite Weyl group contribute to $\pi_0^K$.
On the other hand, if the Weyl group $W L$ is finite, then 
the group $\pi_0^{W L}( \Sigma^\infty_+ ( E W L \times Y(U)^L) )$
is free abelian with a basis given by the set
$W L \bs \pi_0(  Y(U)^L )$,
the $W L$-orbit set of the path components of $Y(U)^L$.

Since colimits commute among themselves, we conclude that
\begin{align*}
 \pi_0^K\left(\Sigma^\infty_+ Y \right)\ &\iso \
\colim \, \pi_0^K\left(\Sigma^\infty_+ Y(U) \right)\ \iso \
 \colim  \bigoplus_{(L)}  \mZ\{ W L\bs \pi_0(  Y(U)^L ) \}
\\
&\iso \
\bigoplus_{(L)}  \mZ\{ W L\bs   (\colim  \pi_0(  Y(U)^L )) \}\
= \ 
\bigoplus_{(L)}  \mZ\{ W L\bs  \pi_0^L(Y) \} \ , 
\end{align*}
where all colimits are over the poset~$s(\Uc_G)$
and the sums are indexed by conjugacy classes with finite Weyl groups.
To verify that this composite isomorphism
takes the class $\tr_L^K(\sigma^L(x))$ in~$ \pi_0^K\left(\Sigma^\infty_+ Y \right)$
to the basis element corresponding to the orbit of~$x\in \pi_0^L(Y)$ 
in the summand indexed by~$L$, one needs to recall the definition
of the isomorphism in tom Dieck's splitting
from~\cite{tomDieck-OrbittypenII}; we omit this.
\end{proof}

Now we prove a representability result.
The {\em stable tautological class} 
\[ e_{G,V} \ = \ \sigma^G(u_{G,V})\ \in \ \pi_0^G(\Sigma^\infty_+ \bL_{G,V}) \]
arises from the unstable tautological class 
$u_{G,V}$ defined in~\eqref{eq:tautological_class}
by applying the stabilization map~\eqref{eq:sigma_map};
so it is represented by the $G$-map
\[  S^V \ \to \ (\bL(V,V)/G)^+\sm S^V\ = \ (\Sigma^\infty_+ \bL_{G,V})(V) \ , \quad 
v\ \longmapsto (\Id_V\cdot G)\sm v  \ .\]

\begin{prop}\label{prop-B_gl represents}  
Let $G$ and $K$ be compact Lie groups and $V$ a faithful $G$-representation.
Then evaluation at the stable tautological class is an isomorphism
\[  \bA(G,K) \ \xra{\ \iso\ } \ \pi_0^K(\Sigma^\infty_+ \bL_{G,V}) \ , \quad
\tau\ \longmapsto \ \tau(e_{G,V})\]
to the 0-th $K$-equivariant homotopy group 
of the orthogonal spectrum $\Sigma^\infty_+ \bL_{G,V}$.
Hence the morphism
\[ \bA(G,-)\ \to \ \upi_0(\Sigma^\infty_+ \bL_{G,V} ) \]
classified by the stable tautological class $e_{G,V}$ 
is an isomorphism of global functors.  
\end{prop}
\begin{proof}
We show first that every natural transformation $\tau:\pi_0^G\to \pi_0^K$ 
is determined by the element~$\tau(e_{G,V})$.
We let $X$ be any orthogonal spectrum and $x\in\pi_0^G(X)$
a $G$-equivariant homotopy class.
Without loss of generality the class $x$ is represented by a continuous based $G$-map
\[ f \ : \ S^{V\oplus W}\ \to \ X(V\oplus W) \]
for some $G$-representation $W$. 
This $G$-map is adjoint to a morphism of orthogonal spectra 
\[ \hat f \ : \ \Sigma^\infty_+ \bL_{G,V\oplus W} \ \to \ X 
\text{\qquad that satisfies\qquad}
 \hat f_*(e_{G,V\oplus W}) \ = \ x \text{\quad in\quad} \pi_0^G(X) \ .\]
We consider the morphism of orthogonal spaces 
$r:\bL_{G,V\oplus W}\to \bL_{G,V}$
that restricts a linear isometry from~$V\oplus W$ to~$V$.
The relation
\[ \pi_0^G(r)(u_{G,V\oplus W}) \ = \ u_{G,V} \]
shows that the composite
\[  \Rep(K,G)\ \xra{[\alpha]\mapsto \alpha^*(u_{G,V\oplus W})} \ 
\pi_0^K(\bL_{G,V\oplus W})\ \xra{\pi_0^K(r)} \ \pi_0^K(\bL_{G,V}) \]
is evaluation at the class $u_{G,V}$.
Evaluation at~ $u_{G,V\oplus W}$ and at~$u_{G,V}$
are both bijective by Theorem~\ref{thm-pi_0 of L_G},
so~$\pi_0^K(r)$ is bijective for all compact Lie groups~$K$.
By Proposition~\ref{prop-pi_0 of Sigma^infty}, 
the induced morphism of suspension spectra
\[ \Sigma^\infty_+ r\ : \ 
\Sigma^\infty_+ \bL_{G,V\oplus W}\ \to \ \Sigma^\infty_+ \bL_{G,V} \]
thus induces an isomorphism on $\pi_0^K(-)$ for all compact Lie groups~$K$,
and it sends $e_{G,W\oplus V}$ to $e_{G,V}$.
The diagram
\[ \xymatrix@C=25mm{
\pi_0^G(\Sigma^\infty_+ \bL_{G,V})\ar[d]_\tau &
\pi_0^G(\Sigma^\infty_+ \bL_{G,V\oplus W})\ar[d]_\tau 
\ar[l]_-{(\Sigma^\infty_+ r)_*}^-\iso\ar[r]^-{\hat f_*} &
\pi_0^G(X)\ar[d]^\tau \\
\pi_0^K(\Sigma^\infty_+ \bL_{G,V}) &
\pi_0^K(\Sigma^\infty_+ \bL_{G,W\oplus V}) \ar[l]^-{(\Sigma^\infty_+ r)_*}_-\iso
\ar[r]_-{\hat f_*} & \pi_0^K(X)} \]
commutes and the two left horizontal maps are isomorphisms.
Since 
\[ x \ = \ \hat f_*((\Sigma^\infty_+ r)_*^{-1}(e_{G,V})) \ ,\]
naturality yields that
\[   \tau(x) \ = \ \tau(\hat f_*((\Sigma^\infty_+ r)_*^{-1}(e_{G,V})))\ = \ 
\hat f_*((\Sigma^\infty_+ r)_*^{-1}(\tau(e_{G,V})))\ .  \]
So the transformation $\tau$ is determined by the value~$\tau(e_{G,V})$.

It remains to construct, for every element 
$y\in \pi_0^K(\Sigma^\infty_+ \bL_{G,V})$, a natural transformation
$\tau:\pi_0^G\to \pi_0^K$ with $\tau(e_{G,V})=y$.
The previous paragraph dictates what to do: 
we represent a given class $x\in\pi_0^G(X)$  by
a continuous based $G$-map $f:S^{V\oplus W}\to X(V\oplus W)$ as above and set
\[   \tau(x) \ = \ \hat f_*((\Sigma^\infty_+ r)_*^{-1}(y)))\ .\]
We omit the verification that the element $\tau(x)$
only depends on the class $x$ and that~$\tau $ is indeed natural.
\end{proof}

We show now that restriction and transfer maps generate all natural operations 
between the 0-dimensional equivariant homotopy group functors
for orthogonal spectra.
Given compact Lie groups~$K$ and $G$, we consider pairs $(L,\alpha)$ consisting of
\begin{itemize}
\item a closed subgroup $L\leq K$ whose Weyl group $W_K L$ is finite, and
\item  a continuous group homomorphism $\alpha:L\to G$.
\end{itemize}  
The {\em conjugate} of $(L,\alpha)$ by a pair of group elements
$(k,g)\in K\times G$  is the pair
$(^k L,\,c_g\circ \alpha\circ c_k)$ consisting of the
conjugate subgroup $^k L$ and the composite homomorphism
\[  {^k L} \ \xra{\ c_k\ }\ L \ \xra{\ \alpha\  }\ 
G \ \xra{\ c_g\ } \ G \ . \]
Since inner automorphisms induce the identity on equivariant homotopy groups,
\[ \tr_{^k L}^K \circ\, (c_g\circ \alpha \circ c_k)^*\ = \ 
 \tr_L^K \circ\, \alpha^* \ : \ \pi_0^G(X) \ \to \ \pi_0^K(X)\ , \]
i.e., conjugate pairs define the same operation on equivariant homotopy groups.

\begin{theorem}\label{thm-Burnside category basis} 
Let $G$ and $K$ be compact Lie groups.
\begin{enumerate}[\em (i)]
\item 
Let $V$ be a faithful $G$-representation. 
Then the homotopy group $\pi_0^K(\Sigma^\infty_+ \bL_{G,V})$
is a free abelian group with basis given by the classes
\[ \tr_L^K(\alpha^*(e_{G,V})) \]
as  $(L,\alpha)$ runs over a set of representatives of
all $(K\times G)$-conjugacy classes of pairs consisting of a closed subgroup~$L$ of $K$ 
with finite Weyl group and a continuous homomorphism $\alpha:L\to G$. 
\item
The morphism group $\bA(G,K)$ 
in the global Burnside category is a free abelian group 
with basis the operations~$\tr_L^K\circ \alpha^*$,
where $(L,\alpha)$ runs over all conjugacy classes of pairs consisting of
a closed subgroup~$L$ of~$K$ with finite Weyl group 
and a continuous homomorphism $\alpha:L\to G$. 
\end{enumerate}
\end{theorem}
\begin{proof}
(i) The map
\[  \Rep(K,G) \ \to \ \pi_0^K(\bL_{G,V})\ , \quad
[\alpha:K\to G] \ \longmapsto \ \alpha^*(u_{G,V})\]
is bijective according to Theorem~\ref{thm-pi_0 of L_G}.
Proposition~\ref{prop-pi_0 of Sigma^infty} thus says that
$\pi_0^K(\Sigma^\infty_+ \bL_{G,V})$
is a free abelian group with a basis given by the elements
\[ \tr_L^K(\sigma^L(\alpha^*(u_{G,V})))\ = \ 
 \tr_L^K(\alpha^*(\sigma^G(u_{G,V})))\ = \  \tr_L^K(\alpha^*(e_{G,V})) \]
where $L$ runs through all conjugacy classes of closed subgroups of $K$ 
with finite Weyl group and $\alpha$ runs through a set of representatives of the 
$W_KL$-orbits of the set $ \Rep(L,G)$.
The claim follows because $(K\times G)$-conjugacy classes of such pairs
$(L,\alpha)$ biject with pairs 
consisting of a conjugacy class of subgroups $(L)$ and
a $W_K L$-equivalence class in $\Rep(L,G)$.

(ii) We let $V$ be any faithful $G$-representation.
By part~(i) the composite
\[ \mZ\{[L,\alpha]\ |\ \ |W_KL| <\infty,\, \alpha:L\to G\}
\ \to \ \text{Nat}(\pi_0^G,\pi_0^K) 
\ \xra{\ \ev\ } \ \pi_0^K(\Sigma^\infty_+ \bL_{G,V})\]
is an isomorphism, where the first  map takes 
a conjugacy class~$[L,\alpha]$ to $\tr_L^K\circ\,\alpha^*$,
and the second map is evaluation at the
stable tautological class $e_{G,V}$.
The evaluation map is an isomorphism by 
Proposition~\ref{prop-B_gl represents}, so the first map
is an isomorphism, as claimed.  
\end{proof}

Theorem~\ref{thm-Burnside category basis}~(ii) is almost a complete calculation
of the global Burnside category, but one important piece of information is still missing:
how does one express the composite of two operations, each given in
the basis of Theorem~\ref{thm-Burnside category basis}, as a sum of basis elements?
Restrictions are contravariantly functorial 
and transfers are transitive, i.e., for every closed subgroup $K$ of $H$ we have
\[  \tr_H^G\circ\tr_K^H \ = \ \tr_K^G \ : \ \pi_0^K(X) \ \to \ \pi_0^G(X) \ . \]
So the key question is how to express a transfer 
followed by a restriction in terms of the specified basis.

Every group homomorphism is the composite of an epimorphism and a subgroup
inclusion.
Transfers commute with inflation (i.e., restriction along epimorphisms):
for every surjective continuous group homomorphism
$\alpha:K\to G$ and every subgroup $H$ of $G$ the relation
\[  \alpha^*\circ \tr_H^G \ = \ \tr_L^K \circ (\alpha|_L)^*    \]
holds as maps $\pi_0^H(X)\to \pi_0^K(X)$, where $L=\alpha^{-1}(H)$ and
$\alpha|_L:L\to H$ is the restriction of $\alpha$.
So the remaining issue is to rewrite the composite 
\[ \pi_0^H(X) \ \xra{\ \tr_H^G\ }\  \pi_0^G(X) \ \xra{\ \res^G_K\ }\  \pi_0^K(X)  \]
of a transfer map and a restriction map,
where~$H$ and~$K$ are two closed subgroups of a compact Lie group $G$.
The answer is given by the {\em double coset formula}:
\begin{equation}\label{eq:double_coset}
   \res^G_K\circ \tr_H^G \ = \ \sum_{[M]}\
\chi^\sharp(M)\cdot \tr_{K\cap{^g H}}^K  \circ c_g^* \circ \res^H_{K^g\cap H}\ .
\end{equation}
The double coset formula was proved by Feshbach for Borel cohomology theories~\cite[Thm.\,II.11]{feshbach} and later generalized to equivariant cohomology theories by 
Lewis and May~\cite[IV \S 6]{lms}. 
The sum in the double coset formula~\eqref{eq:double_coset}
runs over all connected components $M$ of orbit type manifolds,
the group element~$g\in G$ that occurs is such that $K g H\in M$,
and $\chi^\sharp(M)$ is the internal Euler characteristic of~$M$.
Only finitely many of the orbit type manifolds are non-empty,
so the double coset formula is a finite sum.

In this paper we only need the double coset formula when~$H$  
has finite index in~$G$, and then~\eqref{eq:double_coset} simplifies.
For any other subgroup $K$ of~$G$ the intersection $K\cap {^g H}$ 
then has finite index in $K$, so only finite index transfers are
involved in the double coset formula.
Since $G/H$ is finite, so is the set $K\bs G/H$
of double cosets, and all orbit type manifold components are points.
So all internal Euler characteristics that occur are~1 and 
the double coset formula specializes to
\[  \res^G_K\circ \tr_H^G \ = \ \sum_{[g]\in K\backslash G/H}\
\tr_{K\cap{^g H}}^K\circ c_g^* \circ \res^H_{K^g\cap H}  \ ; \]
the sum runs over a set of representatives of the $K$-$H$-double cosets.

The explicit description of the groups~$\bA(G,K)$
allows us to relate our notion of global functor
to other `global' versions of Mackey functors, which are typically introduced 
by specifying generating operations and relations between them.
For example, our category of global functors is equivalent to the
category of {\em functors with regular Mackey structure}
in the sense of Symonds~\cite[\S 3]{symonds-splitting}.

\begin{eg}\label{eg-global functor examples} 
We list some explicit examples of global functors; for more details we refer
to~\cite{schwede-global}.

(i) 
For every compact Lie group~$G$, the represented global functor $\bA(G,-)$ 
is realized by the suspension spectrum of a free orthogonal space~$\bL_{G,V}$,
by Proposition~\ref{prop-B_gl represents}.
In the special case $G=e$ of the trivial group we refer to this represented
global functor as the {\em Burnside ring global functor}
and denote it by~$\mA=\bA(e,-)$.
The value $\mA(K)$ at a compact Lie group $K$ is a free abelian group
with basis indexed by conjugacy classes of closed subgroups of~$K$ with finite Weyl group.
When $K$ is finite, then the Weyl group condition is
vacuous and $\mA(K)$ this is naturally isomorphic to the Burnside ring of~$K$,
compare Remark~\ref{rk-A^fin and A^c} below.

The Burnside ring global functor is realized by the sphere spectrum~$\mS$,
given by $\mS(V)=S^V$ with the canonical homeomorphisms $S^V\sm S^W\iso S^{V\oplus W}$
as structure maps.
The equivariant homotopy groups of the sphere spectrum are
thus the equivariant stable stems.
The action on the unit $1\in\pi_0(\mS)$ is an isomorphism of global functors
\[   \mA \ \xra{\ \iso \ } \ \upi_0(\mS) \]
from the Burnside ring global functor to the 0-th homotopy
global functor of the sphere spectrum. For finite groups, this is originally due
to Segal~\cite{segal-ICM}, and for general compact Lie groups 
to tom\,Dieck, as a corollary to his splitting theorem 
(see~Satz~2 and Satz~3 of~\cite{tomDieck-OrbittypenII}).

(ii)
Given an abelian group $M$, the {\em constant global functor}
is given by $\underline{M}(G)=M$ and all restriction maps
are identity maps. The transfer $\tr_H^G:\underline{M}(H)\to\underline{M}(G)$
is multiplication by the Euler characteristic of the homogeneous space~$G/H$. 
In particular, if $H$ is a subgroup of finite index of $G$, then
$\tr_H^G$ is multiplication by the index $[G:H]$.

(iii)
The {\em representation ring global functor}
$\bRU$ assigns to a compact Lie group $G$ the representation ring $\bRU(G)$, 
the Grothendieck
group of finite dimensional complex $G$-representations.
The fact that the representation rings form a global functor is classical
in the restricted realm of finite groups, but somewhat less familiar
for compact Lie groups in general.
The restriction maps $\alpha^*:\bRU(G)\to \bRU(K)$
are induced by restriction of representations along a homomorphism $\alpha:K\to G$.
The transfer $\tr_H^G:\bRU(H)\to \bRU(G)$ along a closed subgroup inclusion
$H\leq G$ is the {\em smooth induction} 
of Segal~\cite[\S\,2]{segal-representation}.
If $H$ has finite index in~$G$, then this induction 
sends the class of an $H$-representation~$V$ to the induced $G$-representation
$\map^G(H,V)$;
in general, induction may send actual representations to virtual representations.
In the generality of compact Lie groups, 
the double coset formula for $\bRU$ was proved by Snaith~\cite[Thm.\,2.4]{Snaith-Brauer}.
We show in~\cite{schwede-global} that the representation ring global functor $\bRU$
is realized by the periodic global $K$-theory spectrum. 

(iv)
Given any generalized cohomology theory $E$ (in the non-equivariant sense),
we can define a global functor $\underline{E}$  by setting
\[   \underline{E}(G) \ = \ E^0(B G) \ , \]
the 0-th $E$-cohomology of a classifying space of the group $G$.
The contravariant functoriality in group homomorphisms
comes from the covariant functoriality of classifying spaces.
The transfer maps for a subgroup inclusion $H\leq G$ 
comes from the stable transfer map 
\[   \Sigma^\infty_+ B G \ \to \  \Sigma^\infty_+ B H \ .\]
The double coset formula was proved 
in this context by Feshbach~\cite[Thm.\,II.11]{feshbach}.
The global functor $G\mapsto E^0(B G)$ is realized by a preferred global homotopy type:
in~\cite{schwede-global} we introduce a `global Borel theory' functor~$b$
from the non-equivariant stable homotopy category to the global stable homotopy
category such that the global functor $\upi_0(b E)$ is isomorphic to $\underline{E}$.
The functor $b$ is in fact right adjoint to the forget functor
from the global stable homotopy to the non-equivariant stable homotopy category.
\end{eg}

\begin{rk}\label{rk-A^fin and A^c}
The full subcategory $\bA^\text{fin}$ of the global Burnside category~$\bA$ spanned 
by {\em finite} groups has a different, more algebraic description, as we shall now
recall. 
This alternative description is often taken as the definition in
algebraic treatments of global functors.
We define an algebraic Burnside category $\bB$ whose objects are all finite groups. 
The abelian group $\bB(G,K)$ 
of morphisms from a group $G$ to $K$ is the Grothendieck
group of finite $K$-$G$-bisets where the right $G$-action is free.
Composition
\[ \circ \ : \ \bB(K,L) \times \bB(G,K) \ \to \ \bB(G,L) \]
is induced by the balanced product over $K$, i.e., it is the
biadditive extension of 
\[ (S,T) \ \longmapsto \ S\times_K T \ .  \]
Here $S$ has a left $L$-action and a commuting free right $K$-action,
whereas $T$ has a left $K$-action and a commuting free right $G$-action.
The balanced product $S\times_K T$ than inherits a left $L$-action from $S$
and a free right $G$-action from $T$.
Since the balanced product is associative up to isomorphism,
this defines a pre-additive category~$\bB$.

An isomorphism of pre-additive categories $\bA^\text{fin}\iso \bB$
is given by the identity on objects and by the group isomorphisms
$\bA^\text{fin}(G,K)\to \bB(G,K)$
sending a basis element $\tr_H^G\circ\,\alpha^*$ to the
class of the right free $K$-$G$-biset
\[ K\times_{(L,\alpha)} G \ = \ (K\times G) \, /\, (kl,g)\sim(k,\alpha(l)g) \ .\]
The category of `global functors on finite groups',
i.e., additive functors from $\bA^\text{fin}$ to abelian groups, is thus equivalent to
the category of {\em inflation functors} in the sense of~\cite[p.\,271]{webb}.
In the context of finite groups, these inflation functors 
and several other variations of the concept `global Mackey functor' 
have been much studied in algebra and representation theory.
\end{rk}

\begin{rk}\label{rk-homotopy properties}
In Remark~\ref{rk-pointset properties} we observed that only very special
kinds of orthogonal $G$-spectra are part of a `global family',
i.e., isomorphic to an orthogonal $G$-spectrum of the form $X_G$
for some orthogonal spectrum $X$. The previous obstructions were
in terms of pointset level conditions, and now we can also isolate
obstructions to `being global' in terms of the Mackey functor homotopy groups
of an orthogonal $G$-spectrum.

If we fix a compact Lie group $G$ and let $H$ run through
all closed subgroups of $G$, then the collection of~$H$-equivariant homotopy groups
$\pi_0^H(X)$ of an orthogonal spectrum $X$ forms a
{\em Mackey functor} for the group~$G$, with respect to
the restriction, conjugation and transfer maps.
One obstruction for a general orthogonal $G$-spectrum~$Y$ 
to `be global', i.e., equivariantly stably equivalent to $X_G$
for some orthogonal spectrum~$X$,
is that the $G$-Mackey functor $H\mapsto \pi_0^H(Y)$ can be extended to
a global functor.

An extension of a  $G$-Mackey functor to a global functor requires
us to specify values for groups that are not subgroups of $G$,
but it also imposes restrictions on the existing data.
In particular, the $G$-Mackey functor homotopy groups
can be complemented by restriction maps 
along arbitrary group homomorphisms between
the subgroups of $G$. As the extreme case this includes a restriction map
$p^*:\pi_*^e(X)\to \pi_*^G(X)$ associated to the unique homomorphism $p:G\to e$,
splitting the restriction map $\res^G_e:\pi_*^G(X)\to \pi_*^e(X)$. 
So one obstruction to being global is
that~$\res^G_e$ must be a splittable epimorphism.

Another point is that for an orthogonal spectrum~$X$
(as opposed to a general orthogonal $G$-spectrum),
the action of the Weyl group $W_G H$ on $\pi_0^H(X)$ factors 
through the outer automorphism group of $H$. 
In other words, if $g$ centralizes $H$, then $c_g^*$ is the identity of $\pi^H_0(X)$. 
The most extreme case of this is when $H=e$ is the trivial subgroup of $G$.
Every element of $G$ centralizes $e$, so for $G$-spectra 
of the form $X_G$, the conjugation maps on the value at the trivial subgroup
are all identity maps.
\end{rk}

\section{The global homotopy type of symmetric products}

Now we start the  equivariant analysis of the symmetric product filtration.
The main result is a global homotopy pushout square
of orthogonal spectra~\eqref{eq-ideal square},
showing that~$Sp^n$ can be obtained from~$Sp^{n-1}$ by coning off, 
in the global stable homotopy category, a certain morphism
from the suspension spectrum of~$B_{\gl}\Fc_n$.
Non-equivariantly, such a homotopy pushout square
was exhibited by Lesh,
see Theorem~1.1 and Proposition~7.4 of~\cite{lesh-filtration}.
In Theorem~\ref{thm-pi_0 Sp^n} we then 
exploit the Mayer-Vietoris sequence 
of the global homotopy pushout square to
calculate the global functors~$\upi_0(Sp^n)$ inductively.

We define an orthogonal space $C(B_{\gl}\Fc_n)$ by
\[ (C(B_{\gl}\Fc_n))(V) \ = \   D(V,n)/\Sigma_n\ ,  \]
where
\[ D(V,n)\ = \ \left\{ (v_1,\dots,v_n)\in V^n \ : \ 
\sum_{i=1}^n v_i=0\ , \ \sum_{i=1}^n |v_i|^2 \leq  1 \right\} \]
is the unit disc in the kernel of summation map.
Since a unit disc is the cone on the unit sphere, 
$C(B_{\gl}\Fc_n)$ is the unreduced cone of the global classifying space~$B_{\gl}\Fc_n$,
whence the notation. 
Next we define a certain morphism of orthogonal spectra
\[  \Phi\ : \  \Sigma^\infty_+ C(B_{\gl}\Fc_n) \to \  Sp^n  \]
that takes the orthogonal subspectrum $\Sigma^\infty_+ B_{\gl}\Fc_n$ to $Sp^{n-1}$,
and that can be thought of as a highly structured, parametrized Thom-Pontryagin
collapse map. I owe this construction to Markus Hausmann.
Before giving the details we try to explain the main idea.
For every inner product space~$V$, the map~$\Phi(V)$ 
has to assign to each tuple $(y_1,\dots,y_n)\in D(V,n)$
a based map $\Phi(V)(y_1,\dots,y_n):S^V\to Sp^n(V)$ 
that does not depend on the order of $y_1,\dots,y_n$.
We would like to take~$\Phi(V)(y_1,\dots,y_n)$ as the product of the
Thom-Pontryagin collapse maps in balls of sufficiently small
radius centered at the points $y_1,\dots,y_n$.
This would work fine for an individual inner product space~$V$,
but such maps would not form a morphism of orthogonal spectra
as $V$ increases. 

The fix to this problem is to combine the collapse maps with orthogonal
projection onto the subspace spanned by $y_1,\dots,y_n$.
However, this orthogonal projection does {\em not} depend continuously on
the tuple $y$ at those points where the dimension of the span of $y_1,\dots,y_n$ jumps.
So instead of the orthogonal projection to the span we use a certain
positive semidefinite self-adjoint endomorphism $P(y)$ of~$V$ 
that has similar features and varies continuously with~$y$.

\begin{construction}[Collapse maps]
We let~$V$ be an inner product space and denote by~$\sa^+(V)$
the space of positive semidefinite, self-adjoint endomorphisms of~$V$, i.e,
$\mR$-linear maps $F:V\to V$ that satisfy
\begin{itemize}
\item $\td{F(v),v} \geq 0$ for all $v\in V$, and 
\item $\td{F(v),w} = \td{v, F(w)}$ for all $v,w\in V$.
\end{itemize}
We note that~$\sa^+(V)$ is a convex
subset of $\text{End}(V)$, hence contractible. 
We fix the natural number $n\geq 2$ and
set the radius for the collapse maps to
\[ \rho \ = \ \frac{1}{2\cdot n^{3/2}} \ .\]
We define a scaling function
\[ s \ : \ [0,\rho) \ \to \ [0,\infty) \text{\qquad by\qquad} 
s(x)\ = \ x/(\rho-x)\ .\]
What matters is not the precise form of the function~$s$,
but only that it is a homeomorphism from $[0,\rho)$ to~$[0,\infty)$.
We define a {\em parametrized collapse map}
\[ c \ : \ \sa^+(V)\times S^V \ \to \ S^V \]
by
\[ c(F,v) \ = \ 
\begin{cases}
 v\ + \  s(|F(v)|)\cdot F(v)  & \text{ if $v\ne\infty$ and $|F(v)| < \rho$, and}\\
\quad \infty & \text{ else.}
\end{cases}\]
\end{construction}

\begin{lemma}\label{lemma-estimate} 
  \begin{enumerate}[\em (i)]
  \item For all $(F,v)\in \sa^+(V)\times S^V$ the relation~$|c(F,v)|\geq |v|$ holds.  
  \item The map $c$ is continuous.
  \end{enumerate}
\end{lemma}
\begin{proof}
(i) There is nothing to show if $c(F,v)=\infty$, so we may assume that
$v\ne \infty$ and $|F(v)|<\rho$.
Since~$F$ is self-adjoint, $V$ is the orthogonal direct sum of image and kernel of~$F$. 
So we can write
\[ v\ =\ a\ +\ b\ ,  \]
where $a\in\text{im}(F)$, and $b\in\ker(F)$ and~$b$ is orthogonal to~im$(F)$.
The orthogonal decomposition
\begin{align*}
 c(F,v) \ &= \   v\ +\ s(|F(v)|)\cdot F(v)  \  
= \   \big(a +  s(|F(a)|)\cdot F(a) \big)\ +\  b
\end{align*}
allows us to conclude that
\begin{align*}
  |c(F,v)|^2 \ &=  \    \left|a + s(|F(a)|)\cdot F(a) \right|^2 + |b|^2  \\
&= \ |a|^2\ +\ 2 s(|F(a)|)\cdot \td{a, F(a)}  \ + \ s(|F(a)|)^2\cdot |F(a)|^2 \ +\  |b|^2
 \ \geq \ |a|^2\ +\ |b|^2\ = \ |v|^2\ .  
\end{align*}
The inequality uses that~$F$ is positive semidefinite.
Taking square roots proves the claim.

(ii) 
We consider a sequence $(F_k,v_k)$ that converges  in~$\sa^+(V)\times S^V$ to
a point $(F,v)$. 
We need to show that the sequence $c(F_k,v_k)$ converges to~$c(F,v)$ in~$S^V$.  
If $v=\infty$, then $|v_k|$ converges to~$\infty$, 
hence so does $c(F_k,v_k)$ by part~(i). 
So we suppose that~$v\ne\infty$ for the rest of the proof.
Then we can assume without loss of generality that $v_k\ne\infty$ for all~$k$.
We distinguish three cases.

Case 1: $|F(v)|<\rho$.
Then~$|F_k(v_k)|<\rho$ for almost all~$k$.
So~$c(F_k,v_k)$ converges to $c(F,v)$ because the formula 
in the definition of~$c$ is continuous in both parameters~$F$ and~$v$.
 
Case 2: $|F(v)|=\rho$.
If $|F_k(v_k)|\geq\rho$, then $c(F_k,v_k)=\infty$.
Otherwise
\begin{align*}
 |c(F_k,v_k)|\ &= \ \left|v_k\ + \   s(|F_k(v_k)|)\cdot F_k(v_k) \right| \
\geq \   s(|F_k(v_k)|)\cdot|F_k(v_k)|\ -\  | v_k| \ .
\end{align*}
The sequences $|v_k|$ and $|F_k(v_k)|$ converge 
to the finite numbers $|v|$ respectively $\rho$;
on the other hand, $s(|F_k(v_k)|)$ converges to~$\infty$.
So the sequence $|c(F_k,v_k)|$ also converges to~$\infty$,
which means that the sequence~$c(F_k,v_k)$
converges to~$\infty=c(F,v)$.

Case 3: $|F(v)|>\rho$.
Then $|F_k(v_k)|>\rho$ for almost all~$k$.
So~$c(F_k,v_k)=\infty$ for almost all~$k$, and this sequence
converges to~$c(F,v)=\infty$.
\end{proof}

\begin{construction}
We define a continuous map 
\[ P \ : \ V^n \ \to \ \sa^+(V) \text{\qquad by\qquad}
P(y)(v)\ = \ P(y_1,\dots,y_n)(v)\ = \ \sum_{j=1}^n \td{v,y_j}\cdot y_j\ .\]
Each of the summands $\td{-,y_j}\cdot y_j$ is self-adjoint, hence
so is $P(y)$. Because
\[ \td{P(y)(v),v} \ = \  \sum_{j=1}^n \td{v,y_j}^2\ \geq \ 0 \]
the map $P(y)$ is also positive semidefinite.
If the family $(y_1,\dots,y_n)$ happens to be orthonormal, then
$P(y)$ is the orthogonal projection onto the span of $y_1,\dots,y_n$.
In general, $P(y)$ need not be idempotent, 
but its image is always the span of the vectors $y_1,\dots,y_n$,
and hence its kernel is the orthogonal complement of that span.

For every linear isometric embedding $\varphi:V\to W$ 
and an endomorphism $F\in \sa^+(V)$, we define $^\varphi F\in\sa^+(W)$ by
`conjugation and extension by~0', i.e., we set
\[ (^\varphi F)(\varphi(v)+w)\ = \  \varphi(F(v)) \]
for all $(v,w)\in V\times\varphi^\perp$. Then
\begin{equation}\label{eq:varphi_past_P}
  ^\varphi ( P(y))(\varphi(v)+w)\ = \ 
  \varphi( P(y)(v))\ = \  \sum_{j=1}^n \td{\varphi(v),\varphi(y_j)}\cdot \varphi(y_j)\
= \  P(\varphi(y))(\varphi(v)+w)\ ,
\end{equation}
i.e., ${^\varphi}( P(y))=  P(\varphi(y))$ as endomorphisms of~$W$.

It will be convenient to extend the meaning of the minus symbol
and allow to subtract a vector from infinity.
We define a continuous map
\[ \ominus\ : \ S^V\times V\ \to\  S^V \text{\qquad by\qquad}
 v \ominus z \ = \  
\begin{cases}
v - z & \text{\ for $v\ne\infty$, and}\\  
\quad \infty& \text{\ for $v=\infty$.}
\end{cases} \]
We emphasize that only the first argument of the operator $\ominus$
is allowed to be infinity; in particular, 
we cannot subtract~$\infty$ from itself.
We define a continuous map
\[ \tilde\Phi(V) \ : \ D(V,n)\times S^V \ \to \ Sp^n(S^V) \text{\qquad by\qquad}
 \tilde\Phi(V)(y,v) \ = \ 
[c(P(y),v\ominus y_1),\dots,c(P(y),v\ominus y_n)] \ .\]
The map~$\tilde\Phi(V)$ sends $D(V,n)\times \{\infty\}$ to the basepoint.
For every permutation~$\sigma\in\Sigma_n$ we have
$P(y\cdot\sigma)=P(y)$ and hence
\[ \tilde\Phi(V)(y\cdot\sigma,v)\ = \ \tilde\Phi(V)(y,v)  \ .\]
So~$\tilde\Phi(V)$ factors over a continuous map
\[ \Phi(V)\ : \  (C(B_{\gl}\Fc_n)(V))_+\sm S^V \ = \ 
(D(V,n)/\Sigma_n)_+ \sm S^V \ \to \ Sp^n(S^V)\ .\]
\end{construction}

\begin{lemma}
As~$V$ varies over all inner product spaces,  
the maps $\Phi(V)$ form a morphism of orthogonal spectra
\[  \Phi\ : \ \Sigma^\infty_+ C(B_{\gl}\Fc_n)\ \to \  Sp^n \ . \]
\end{lemma}
\begin{proof}
For every linear isometric embedding $\varphi:V\to W$,
every $F\in\sa^+(V)$ and all~$(v,w)\in V\times \varphi^\perp$
with $|F(v)|<\rho$ we have
\begin{align}\label{eq-c-phi-relation}
  c(^\varphi F,\varphi(v)+w)\ &= \ 
 (\varphi(v)+w)\ + \ s(|(^\varphi F)(\varphi(v)+w)|)\cdot (^\varphi F)(\varphi(v)+w) \nonumber\\
&= \ 
 \varphi(v) + w \ +  \  s(|\varphi( F(v))|)\cdot \varphi(F(v)) 
\ = \ \varphi(c(F,v)) + w 
\end{align}
in~$S^W$.
Hence for all $y\in D(V,n)$,
\begin{align*}
c(P(\varphi(y)), (\varphi(v)+w )\ominus \varphi(y_i)) 
\ &=_\eqref{eq:varphi_past_P}  \  
c\left(^\varphi(P(y)), ( \varphi(v)+ w)\ominus \varphi(y_i) \right) \\
\ &= \qquad
c\left(^\varphi(P(y)),  \varphi(v\ominus y_i)+ w \right) \\
&=_\eqref{eq-c-phi-relation}\ 
\varphi(c(P(y),v\ominus y_i)) + w \ .
\end{align*}
This shows that the square
\[ 
 \xymatrix@C=5mm{
\bO(V,W)\times D(V,n)\times  S^V \ar[rrrr]^-{\bO(V,W)\times\tilde\Phi(V)} 
\ar[d] &&&& 
\bO(V,W)\times Sp^n(S^V) \ar[d] & ((w,\varphi),[v_1,\dots,v_n]) \ar@{|->}[d]\\
D(W,n)\times S^W \ar[rrrr]_-{\tilde\Phi(W)} &&&& Sp^n(S^W) & 
[w+\varphi(v_1),\dots,w+\varphi(v_n)]} \]
commutes, where the vertical maps are the structure maps.
\end{proof}

We claim that the morphism~$\Phi$ takes the orthogonal subspectrum
$\Sigma^\infty_+ (B_{\gl}\Fc_n)$ into the subspectrum~$Sp^{n-1}$. 
We will need that kind of argument again later, so we formulate it more generally.

\begin{lemma}\label{lemma-land in Sp^n-1}
Let $V$ be an inner product space and $y\in S(V,n)$.
For $t\in[0,1]$ we define $F_t\in\sa^+(V)$ by $F_t=(1-t)\cdot P(y)+t\cdot\Id_V$.
Then for every $v\in S^V$ the point
\[ [c(F_t,v\ominus y_1),\dots,c(F_t,v\ominus y_n)] \ \in \ Sp^n(S^V)\]
belongs to the subspace~$Sp^{n-1}(S^V)$.
\end{lemma}
\begin{proof}
Since $\sum_{i=1}^n|y_i|^2=1$
there is at least one~$i\in\{1,\dots,n\}$ with $|y_i|^2\geq 1/n$.
The Cauchy-Schwarz inequality gives
\begin{align*}
 |y_i|\cdot |F_t(y_i)| \ \geq \ |\td{y_i, F_t(y_i)}|
\ &= \   t \td{y_i,y_i}\ + \ (1-t)\sum_{j=1}^n \td{y_i,y_j}^2
\ \geq \ 
 t|y_i|^2\ + \ (1-t)|y_i|^4 \ .
\end{align*}
Dividing by $|y_i|$ yields
\begin{align*}
  |F_t(y_i)| \ \geq \  t|y_i|\ + \ (1-t)|y_i|^3 \ \geq 
\  \frac{t}{n^{1/2}}\ + \ \frac{1-t}{n^{3/2}}\ \geq \ \frac{1}{n^{3/2}}\ =\ 2\rho \ .
\end{align*}
The relation
\begin{align*}
  \sum_{j=1}^n | F_t(y_i-y_j)| \ \geq \ 
  \left| \sum_{j=1}^n  F_t(y_i-y_j)\right | \ =  \ 
  \left|n\cdot F_t(y_i)- F_t(y_1+\dots+y_n)\right|\ = \     n |F_t(y_i) | 
\end{align*}
shows that there is a $j\in\{1,\dots,n\}$ such that
\[  | F_t(y_i)-F_t(y_j)| \ = \    | F_t(y_i-y_j)| \ \geq \ |F_t(y_i) | \ \geq \ 2\rho\ . \]
So every $v\in V$ has distance at least~$\rho$ from $F_t(y_i)$ or
from $F_t(y_j)$. Hence~$c(F_t,v\ominus y_i)$ or~$c(F_t,v\ominus y_j)$ 
is the basepoint at infinity of $S^V$.
\end{proof}

For $t=0$, Lemma~\ref{lemma-land in Sp^n-1} 
shows that for every $v\in S^V$, the point
\[ \Phi(V)(y\cdot \Sigma_n, v)\ = \ 
[c(P(y),v\ominus y_1),\dots,c(P(y),v\ominus y_n)] \]
belongs to the subspace~$Sp^{n-1}(S^V)$.
So the map $\Phi(V)$ takes the subspace 
$(\Sigma^\infty_+ B_{\gl}\Fc_n)(V)$ to~$Sp^{n-1}(S^V)$. 
We denote by
\[  \Psi\ : \ \Sigma^\infty_+ B_{\gl} \Fc_n\ \to \  Sp^{n-1}  \]
the restriction of the morphism~$\Phi:\Sigma^\infty_+ C(B_{\gl} \Fc_n)\to Sp^n$
to the suspension spectrum of~$B_{\gl}\Fc_n$.
The two vertical maps in the following commutative square~\eqref{eq-ideal square}
are levelwise equivariant h-cofibrations. 
So the following theorem effectively says
that the square is a global homotopy pushout.

\begin{theorem}\label{thm-main homotopy} 
The morphism induced on vertical quotients by the commutative square
of orthogonal spectra
\begin{equation}  \begin{aligned}\label{eq-ideal square}
 \xymatrix{
\Sigma^\infty_+ B_{\gl}\Fc_n \ar[r]^-{\Psi}\ar[d] & Sp^{n-1}\ar[d]\\
\Sigma^\infty_+ C(B_{\gl}\Fc_n) \ar[r]_-\Phi & Sp^n }     
 \end{aligned}\end{equation}  
is a global equivalence.
\end{theorem}
\begin{proof}
We show that for every inner product space~$V$ the map
\[\Phi(V)/\Psi(V) \ : \  
\bar D(V,n)/\Sigma_n \sm S^V \ \to \ Sp^n(S^V)/Sp^{n-1}(S^V)\]
is $O(V)$-equivariantly based homotopic to an equivariant homeomorphism,
where $\bar D(V,n)=D(V,n)/S(V,n)$.
We define continuous maps
\[ G_i\ : \  
\big(   [0,1]\times D(V,n)\ \bs \ \{1\}\times S(V,n) \big) \times S^V  
\ \to \ S^V \]
for $1\leq i\leq n$ by
\[ G_i(t,y,v) \ = \ c\left( (1-t)\cdot P(y),\ v\ominus \frac{y_i}{1- t|y|}\right)  .\]
Here the domain of definition of~$G_i$ is the space of those tuples 
$(t,y,v)\in  [0,1]\times D(V,n) \times S^V$
such that $t<1$ or $|y|<1$.

We claim that the map
\[  (G_1,\dots,G_n)\ : \
 [0,1)\times D(V,n) \times S^V \ \to \ (S^V)^n \]
takes the subspace $[0,1)\times S(V,n) \times S^V$ of the source
into the wedge inside of the product $(S^V)^n$.
Indeed, because $\Phi(V)$ takes $(S(V,n)/\Sigma_n)_+\sm S^V$ into
$Sp^{n-1}(S^V)$, for every $v\in V$ there is an $i\in\{1,\dots,n\}$ with
\[ c(P(y), ((1-t)v)\ominus y_i ) \ = \ \infty \ , \]
i.e., $|P(y)((1-t)v-y_i)|\geq\rho$. Because
\begin{align*}
 \left| (1-t)\cdot P(y)\left(v - \frac{y_i}{1-t}\right)\right| \ = \ 
 | P(y)((1-t)v - y_i)|\  &\geq \  \rho \ ,
\end{align*}
this implies that~$G_i(t,y,v) =\infty$.

We warn the reader that the maps~$G_i$ do {\em not}
extend continuously to $[0,1]\times D(V,n)\times S^V$! 
However, smashing all $G_i$ together remedies this. In other words, we claim that
the map
\[ \big(   [0,1]\times D(V,n)\ \bs \ \{1\}\times S(V,n) \big) \times S^V  
 \ \xra{\ G_1\sm\dots\sm G_n\ }  \ (S^V)^{\sm n} \]
has a continuous extension (necessarily unique)
\[  \bar G \ : \  [0,1]\times D(V,n) \times S^V \ \to \ (S^V)^{\sm n} \]
that sends~$\{1\}\times S(V,n)\times S^V$ to the basepoint.
To prove the claim, we consider any sequence $(t_m,y^m,v^m)_{m\geq 1}$ in
$\big(   [0,1]\times D(V,n)\ \bs \ \{1\}\times S(V,n) \big)\times S^V$ 
that converges to a point $(1,y,v)$ with $|y|=1$.
We claim that there are $i,j\in\{1,\dots,n\}$
such that $|y_i-y_j|\geq 4\rho$.
Indeed, if that were not the case, then we would have
\begin{align*}
2n\ &= \ \left( n \sum_{i=1}^n |y_i|^2\right) - 
2\left\langle\sum_{i=1}^n y_i,\sum_{j=1}^n y_j\right\rangle
+\left(n \sum_{j=1}^n |y_j|^2\right)\\
&= \ \sum_{i,j=1}^n \left( |y_i|^2- 2\td{y_i,y_j} +|y_j|^2\right)\
= \ \sum_{i,j=1}^n |y_i-y_j|^2 \ < \ (4\rho)^2 n^2\ = \ 4/n\ ,
\end{align*}
a contradiction.

Since $\lim_{m\to\infty} y^m=y$, we deduce that $|y_i^m-y_j^m|\geq 2\rho$ 
for all sufficiently large~$m$.
For these~$m$ there is then a~$k\in\{i,j\}$ such that
\[ \left| (1- t_m|y^m|)v^m - y_k^m \right | 
\ \geq \ \rho \ ,\]
and hence
\begin{align*}
 | G_k(t_m,y^m,v^m)| \ 
\ &= \ \left| c\left( (1-t_m)\cdot P(y^m),\ v^m\ominus\frac{y_k^m}{1- t_m|y^m|} \right) \right| \\ 
&\geq \ \left| v^m -\frac{y_k^m}{1- t_m|y^m|} \right| 
\ \geq \ \frac{\rho}{1- t_m|y^m|} \ .  
\end{align*}
The first inequality is Lemma~\ref{lemma-estimate}~(i).
Since the sequences $(t_m)$ and $|y^m|$ converge to~1, the length of the vector
\[ (G_1(t_m,y^m,v^m),\dots,G_n(t_m,y^m,v^m)) \]
tends to infinity with~$m$, so it converges to the basepoint at infinity
of $S^{V^n}=(S^V)^{\sm n}$.

We have now completed the verification that 
the map~$\bar G:[0,1]\times D(V,n) \times S^V\to (S^V)^{\sm n}$
is continuous.
Because the endomorphism $P(y)$ does not depend on the order of the
components of the tuple~$y$, the maps $G_i$ satisfy
\[ G_i(t, y\cdot \sigma, v) \ = \ G_{\sigma(i)}(t,y,v) \ ,\]
so the map $\bar G$ descends to a well-defined continuous and $O(V)$-equivariant map
\[ [0,1]\times \left( \bar D(V,n)/\Sigma_n \sm S^V \right)
\ \to \ (S^V)^{\sm n}/\Sigma_n \ = \ Sp^n(S^V)/Sp^{n-1}(S^V)\ ,\]
which is the desired equivariant homotopy.
This homotopy starts with the map~$\Phi(V)/\Psi(V)$
and ends with the map 
\begin{align*}
 \bar D(V,n)/\Sigma_n \sm S^V \ &\to \qquad Sp^n(S^V)/Sp^{n-1}(S^V)\\
(y\cdot \Sigma_n ,v) \qquad &\longmapsto \ 
\left( v-\frac{y_1}{1-|y|} \right)\sm\dots\sm\left( v-\frac{y_n}{1-|y|}\right)  \ ;
\end{align*}
this map is a continuous bijection from a compact space
to a Hausdorff space, hence a homeomorphism.
\end{proof}

We recall from~\eqref{eq:define_u_n} the definition of
the unstable homotopy  class $u_n\in\ \pi_0^{\Sigma_n} (B_{\gl}\Fc_n)$.
The stabilization map~\eqref{eq:sigma_map} lets us define a $\Sigma_n$-equivariant 
stable homotopy class
\[ w_n \ =  \ \sigma^{\Sigma_n}(u_n) \ \in \ \pi_0^{\Sigma_n}(\Sigma^\infty_+ B_{\gl}\Fc_n) \ . \] 
The last ingredient for our main calculation is to
determine the image of~$w_n$ under the morphism of orthogonal spectra
\[  \Psi\ : \ \Sigma^\infty_+ B_{\gl} \Fc_n\ \to \  Sp^{n-1} \ . \]

\begin{prop}\label{prop:Psi_of_u_n}
The relation
\[ 
 \Psi_*(w_n) \ = \ i_*\left(\tr_{\Sigma_{n-1}}^{\Sigma_n}(1)\right)   
 \]
holds in the group~$\pi_0^{\Sigma_n}(Sp^{n-1})$, where $i:\mS\to Sp^{n-1}$
is the inclusion.
\end{prop}
\begin{proof}
The class $\Psi_*(w_n)$ is represented by the composite $\Sigma_n$-map 
\begin{equation}\label{eq:first_rep}
 S^{\nu_n} \ \xra{(d_1,\dots,d_n)\cdot\Sigma_n\sm -} \
  (S(\nu_n,n)/\Sigma_n)_+\sm S^{\nu_n} \ \xra{\Psi(\nu_n)} \
Sp^{n-1}(S^{\nu_n})  \ ,
\end{equation}
where~$\nu_n$ is the reduced natural $\Sigma_n$-representation and
\[  (d_1,\dots,d_n) \ = \ \frac{1}{\sqrt{n-1}}(b-e_1,\dots,b-e_n) 
\ \in \ S(\nu_n,n) \ .\]
We define an equivariant homotopy to a different map that is easier to
understand.

The space $\sa^+(\nu_n)$ of positive semidefinite self-adjoint endomorphisms of~$\nu_n$ 
is convex, so we can interpolate between 
$P(d_1,\dots,d_n)$ and the identity of~$\nu_n$ in~$\sa^+(\nu_n)$ by the linear homotopy
\[ t\ \longmapsto \ F_t\ = \  (1-t)\cdot P(d_1,\dots,d_n) \ + \ t \cdot \Id_{\nu_n}\ .\]
This induces a homotopy
\[  K \ : \  [0,1]\times S^{\nu_n}\ \to \  Sp^n(S^{\nu_n})  \ , \quad
K(t,v)\ = \ [ c(F_t, v\ominus d_1),\dots, c(F_t, v\ominus d_n) ]\ .\]
For every permutation~$\sigma\in\Sigma_n$ we have $\sigma\cdot d_i=d_{\sigma(i)}$,
hence
\[     {^\sigma}( P(d_1,\dots,d_n))\ =_\eqref{eq:varphi_past_P} \  
P(\sigma\cdot d_1,\dots,\sigma\cdot d_n)
\ =\  P(d_{\sigma(1)},\dots,d_{\sigma(n)})\ =\  
P(d_1,\dots,d_n) \]
as endomorphisms of~$\nu_n$.
Thus also ${^\sigma}( F_t) = F_t$ and hence
\begin{align*}
 \sigma\cdot c(F_t,v\ominus d_i) 
\ =_\eqref{eq-c-phi-relation} \ 
c(^\sigma(F_t), \sigma\cdot(v\ominus d_i)) \ = \ c(F_t, (\sigma\cdot v)\ominus(\sigma\cdot d_i)) \ 
= \ c(F_t, (\sigma\cdot v)\ominus d_{\sigma(i)}) \ ,
\end{align*}
so $\sigma\cdot K(t,v)=K(t,\sigma\cdot v)$,
i.e., the homotopy~$K$ is $\Sigma_n$-equivariant.
A priori, the homotopy takes values in the $n$-th symmetric product;
however, Lemma~\ref{lemma-land in Sp^n-1} applied to~$V=\nu_n$ and $y=(d_1,\dots,d_n)$
shows that~$K(t,v)$
belongs to~$Sp^{n-1}(S^{\nu_n})$.

The homotopy~$K$ starts with the composite~\eqref{eq:first_rep}, 
so the map
\[  K(1,-) \ : \ S^{\nu_n} \ \to \  Sp^{n-1}(S^{\nu_n})  \text{\qquad given by\qquad}
K(1,v)\ = \  \big[c(\Id_V,v\ominus d_1),\dots,c(\Id_V,v\ominus d_n) \big]\]
is another representative of the class $\Psi_*(w_n)$. Because
\[ |d_i-d_j|\ = \ \sqrt{\frac{2}{n-1}} \ \geq \ \frac{1}{n^{3/2}}\ = \  2\rho  \]
for all $i\ne j$, the interiors of the $\rho$-balls around the points 
$d_1,\dots,d_n$ are disjoint. 
So for every~$v\in S^{\nu_n}$ at most one of the points 
$c(\Id_V,v\ominus d_1),\dots, c(\Id_V,v\ominus d_n)$ 
is different from the basepoint of~$S^{\nu_n}$ at infinity.
The map~$K(1,-)$ thus equals the composite
\[ S^{\nu_n} \ \xra{\ J\ }\ S^{\nu_n} \ \xra{\ i\ }\ 
Sp^{n-1}(S^{\nu_n}) \]
where the first map is defined by
\[  J(v)\ = \ 
\begin{cases}
  c(\Id_V,v\ominus d_i) =  \frac{v-d_i}{1-|v-d_i|/\rho}  & \text{ if $v\ne\infty$ and $|v-d_i|<\rho$, and}\\
\quad \infty & \text{ else.}
\end{cases}  \]
The map $J$ is a $\Sigma_n$-equivariant Thom-Pontryagin collapse map
around the image of the equivariant embedding
\[ \Sigma_n/\Sigma_{n-1} \ \to \ \nu_n \ , \quad
\sigma \Sigma_{n-1} \ \longmapsto \ d_{\sigma(n)} \ .\]
So~$J$ represents the class $\tr_{\Sigma_{n-1}}^{\Sigma_n}(1)$
in the equivariant 0-stem~$\pi_0^{\Sigma_n}(\mS)$.
\end{proof}

Now we can put the pieces together and prove our main calculation.
We let~$I_n$ denote the global subfunctor of the Burnside ring 
global functor~$\mA$ generated by the element 
$t_n=n\cdot 1-\tr_{\Sigma_{n-1}}^{\Sigma_n}(1)$ in $\mA(\Sigma_n)$,
and we let~$I_\infty$ denote the union of the global functors~$I_n$ for all~$n\geq 1$.

\begin{theorem}\label{thm-pi_0 Sp^n} 
The inclusion of orthogonal spectra $Sp^{n-1}\to Sp^n$
induces an epimorphism 
\[  \upi_0(Sp^{n-1}) \ \to \ \upi_0(Sp^n)\]
of the 0-th homotopy global functors
whose kernel is generated, as a global functor, by the class
\[ i_*\left(   n \cdot 1\ -\ \tr_{\Sigma_{n-1}}^{\Sigma_n}(1) \right) 
\ \in \ \pi_0^{\Sigma_n}(Sp^{n-1})\ , \]
where $i:\mS\to Sp^{n-1}$ is the inclusion.
For every $n\geq 1$ and $n=\infty$ the action of the Burnside ring global functor 
on the class $i_*(1)\in \pi_0^e(Sp^n)$ passes to an isomorphism 
of global functors
\[ \mA/I_n \ \iso \ \upi_0(Sp^n) \ .\]  
\end{theorem}
\begin{proof}
For every inner product space~$V$ the embedding $Sp^{n-1}(S^V)\to Sp^n(S^V)$ 
has the $O(V)$-equivariant homotopy extension property.
The cone inclusion $q:B_{\gl}\Fc_n \to C(B_{\gl}\Fc_n)$
also has the levelwise homotopy extension property, hence so does
the induced morphism 
$j=\Sigma^\infty_+ q : \Sigma^\infty_+ B_{\gl}\Fc_n \to\Sigma^\infty_+ C(B_{\gl}\Fc_n)$ 
of suspension spectra.
So Theorem~\ref{thm-main homotopy} says that the commutative square
of orthogonal spectra~\eqref{eq-ideal square} 
is a global homotopy pushout square. 
Taking equivariant stable homotopy groups thus results in an exact
Mayer-Vietoris sequence 
that ends in the exact sequence of global functors
\begin{equation}\label{eq:MV-sequence}
 \upi_0( \Sigma^\infty_+ B_{\gl}\Fc_n ) \ \xra{\genfrac(){0pt}{}{j_*}{\Psi_*}}
\ \upi_0( \Sigma^\infty_+ C(B_{\gl}\Fc_n) )\oplus\upi_0( Sp^{n-1}) \ \xra{(-\Phi_*,\text{incl}_*)} \ \upi_0( Sp^n) \ \to \ 0 \ .
\end{equation}

 By Proposition~\ref{thm-u_n generates T_n} the Rep-functor
 $\upi_0 (B_{\gl}\Fc_n)$ is generated by the element $u_n$
 in $\pi_0^{\Sigma_n} (B_{\gl}\Fc_n)$. So by Proposition~\ref{prop-pi_0 of Sigma^infty} 
 the global functor $\upi_0 (\Sigma^\infty_+ B_{\gl}\Fc_n)$
 is generated by the element $w_n=\sigma^{\Sigma_n}(u_n)$
 in $\pi_0^{\Sigma_n} (\Sigma^\infty_+ B_{\gl}\Fc_n)$. 
 The orthogonal space~$C(B_{\gl}\Fc_n)$ is contractible; so 
 its suspension spectrum~$\Sigma^\infty_+ C(B_{\gl}\Fc_n)$ 
 is globally equivalent to the sphere spectrum.
 Thus~$\upi_0(\Sigma^\infty_+ C(B_{\gl}\Fc_n))$ is isomorphic to the Burnside
 ring global functor~$\mA$, and it is freely generated by the class
 $1=\sigma^e(u)$ in~$\pi_0^e(\Sigma^\infty_+ C(B_{\gl}\Fc_n))$,
 where  $u\in\pi_0^e(C(B_{\gl}\Fc_n))$ is the unique element.
 We record that
 \begin{align*}
   j_*(w_n) \ &= \    (\Sigma^\infty_+ q)_*(\sigma^{\Sigma_n}(u_n)) 
   \ = \    \sigma^{\Sigma_n}(q_*(u_n)) \
   = \   \sigma^{\Sigma_n}(p^*(u)) 
   \ = \   p^*(\sigma^e(u))\ = \ p^*(1)\ , 
 \end{align*}
 where $p:\Sigma_n\to e$ is the unique homomorphism.
 The relation~$q_*(u_n)=p^*(u)$ holds because 
 the set~$\pi_0^{\Sigma_n}(C(B_{\gl}\Fc_n))$ has only one element.
 
 Since the global functor $\upi_0(\Sigma^\infty_+ C(B_{\gl}\Fc_n))$ is freely generated
 by the class~$1$, 
 there is a unique morphism 
 $s:\upi_0(\Sigma^\infty_+ C(B_{\gl}\Fc_n))\to \upi_0( \Sigma^\infty_+ B_{\gl}\Fc_n )$
 such that $s(1)=\res^{\Sigma_n}_e(w_n)$. Then
 \[  j_*(s(1))\ = \   \res^{\Sigma_n}_e(j_*(w_n))
 \ = \ \res^{\Sigma_n}_e(p^*(1))\ =  \ 1 \ . \]
 Morphisms out of the global functor~$\upi_0(\Sigma^\infty_+ C(B_{\gl}\Fc_n))$
 are determined by their effect on the universal class, 
 so we conclude that $j_*\circ s=\Id$.
 In particular, $j_*$ is an epimorphism and
 the Mayer-Vietoris sequence~\eqref{eq:MV-sequence}
 restricts to an exact sequence of global functors
 \[ \ker(j_*) \ \xra{\ \Psi_*\ }
 \ \upi_0( Sp^{n-1}) \ \xra{\text{incl}_*} \ \upi_0( Sp^n) \ \to \ 0 \ . \]
 The other composite $s\circ j_*$  is 
 an idempotent endomorphism of $\upi_0(\Sigma^\infty_+ B_{\gl}\Fc_n)$,
 and it satisfies
 \[   s( j_*(w_n))\ = \    s(p^*(1))\ = \  p^*(s(1)) \ = \  p^*(\res^{\Sigma_n}_e(w_n))\ . \]
 The kernel of $j_*$ is thus generated as a global functor by 
 \[  (s\circ j_*-\Id)(w_n) \ = \ p^*(\res^{\Sigma_n}_e(w_n))\ - \ w_n \ . \]
 The global functor $\Psi_*(\ker(j_*))$ is then generated by the class
 \begin{align*}
   \Psi_*( p^*(\res^{\Sigma_n}_e(w_n))- w_n )\ 
   &= \
   p^*\left(\res^{\Sigma_n}_e\left(i_*\left(\tr_{\Sigma_{n-1}}^{\Sigma_n}(1)\right)\right)\right) 
   \ - \    i_*\left(\tr_{\Sigma_{n-1}}^{\Sigma_n}(1)\right) \\
   &= \
   i_*\left(p^*\left(\res^{\Sigma_n}_e\left(\tr_{\Sigma_{n-1}}^{\Sigma_n}(1)\right)\right)- \tr_{\Sigma_{n-1}}^{\Sigma_n}(1)\right)\ 
   = \ i_* \left( n \cdot 1 - \tr_{\Sigma_{n-1}}^{\Sigma_n}(1) \right)  \ .
 \end{align*}
 The first equality uses Proposition~\ref{prop:Psi_of_u_n}. 
 This proves the first claim.
 
 The second claim is then obtained by induction over~$n$,
 using that $I_{n-1}\subset I_n$.
 For $n=\infty$ we use that the canonical map
 \[ \colim_{n} \, \upi_0(Sp^n) \ \to \ \upi_0(Sp^\infty)\]
 is an isomorphism because each embedding $Sp^{n-1}\to Sp^n$
 is levelwise an equivariant h-cofibration.
\end{proof}

\section{Examples}\label{sec-examples}

In this last section we make the description of the global functor
$\upi_0(Sp^n)$ of Theorem~\ref{thm-pi_0 Sp^n} 
more explicit by exhibiting a generating set for the group $I_n(G)$, the kernel of
the map $\mA(G)\iso\pi_0^G(\mS)\to\pi_0^G(Sp^n)$,
in terms of the subgroup structure of~$G$.
We use this to determine~$\pi_0^G(Sp^n)$, for all $n$,
when~$G$ is a~$p$-group, a symmetric group $\Sigma_k$ for $k\leq 4$,
and the alternating group~$A_5$.
The purpose of these calculations is twofold: we want to illustrate that
$\pi_0^G(Sp^n)$ can be worked out explicitly in terms of the poset
of conjugacy classes of subgroups of~$G$ and their relative indices;
and we want to convince the reader that the explicit answer for the
group~$\pi_0^G(Sp^n)$ is much less natural than the
global description of~$\upi_0(Sp^n)$ given by Theorem~\ref{thm-pi_0 Sp^n}.

For a pair of closed subgroups $K\leq H$ of a compact Lie group $G$ 
such that~$K$ has finite index in~$H$ 
we denote by $t^H_K \in \mA(G)$ the class 
\[ t^H_K\ = \ [H:K]\cdot \tr_H^G(1)\ -\ \tr_K^G(1)  \ .\]
For example, $t_n=t^{\Sigma_n}_{\Sigma_{n-1}}$.
The notation is somewhat imprecise because it does not record the ambient group~$G$,
but that should always be clear from the context.
In the next proposition these classes feature under the hypothesis that
the Weyl group of $H$ in~$G$ is finite (so that $\tr_H^G(1)$ is a non-trivial
class in the Burnside ring of~$G$). However, the group~$K$
may have infinite Weyl group, in which case $\tr_K^G(1)=0$ 
and~$t^H_K$ simplifies to $[H:K]\cdot \tr_H^G(1)$.

\begin{prop}\label{prop-describe I_n} 
For every $n\geq 2$ and every compact Lie group~$G$, 
the abelian group $I_n(G)$ is generated by the classes $t_K^H$
as $(H,K)$ runs through a set of representatives of all $G$-conjugacy classes
of nested pairs $K\leq H$ of closed subgroups of~$G$ such that
\begin{itemize}
\item $[H:K]\leq n$ and 
\item the Weyl group $W_G H$ is finite. 
\end{itemize}
\end{prop}
\begin{proof}
By definition $I_n$ is the image of the morphism of global functors
\[ \bA(\Sigma_n,-) \ \to \ \mA \]
represented by $t_n\in \mA(\Sigma_n)$. 
By Theorem~\ref{thm-Burnside category basis}~(ii) 
the group $\bA(\Sigma_n,G)$ is generated by the operations
$\tr_H^G\circ\alpha^*$ where $(H,\alpha)$ runs through the $(G\times\Sigma_n)$-conjugacy
classes of pairs consisting of a closed subgroup $H\leq G$ with finite Weyl group
and a continuous homomorphism $\alpha:H\to \Sigma_n$.
So~$I_n(G)$ is generated, as an abelian group, by the classes
\[ \tr_H^G(\beta^*(\res^{\Sigma_n}_\Gamma(t_n)))\ \in \ \mA(G) \ ,\]
where $\Gamma$ is a subgroup of~$\Sigma_n$ and $\beta:H\to \Gamma$
a continuous epimorphism.
The double coset formula (in the finite index case) gives
\begin{align*}
 \res^{\Sigma_n}_\Gamma(t_n) \ &= \   
n\cdot 1 - \res^{\Sigma_n}_\Gamma(\tr_{\Sigma_{n-1}}^{\Sigma_n}(1)) \\ 
&= \ 
 n\cdot 1 \ - \sum_{[\sigma]\in \Gamma\backslash \Sigma_n/\Sigma_{n-1}} 
\tr_{\Gamma\cap{^\sigma \Sigma_{n-1}}}^\Gamma(1)
= \ \sum_{[\sigma]\in \Gamma\backslash \Sigma_n/\Sigma_{n-1}} 
t^\Gamma_{\Gamma\cap {^\sigma \Sigma_{n-1}}}
\end{align*}
in~$\mA(\Gamma)$, where we used that
\[ \sum_{[\sigma]\in \Gamma\backslash \Sigma_n/\Sigma_{n-1}}
[\Gamma:\Gamma\cap {^\sigma \Sigma_{n-1}}]  \ = \  n \ .\]
Thus
\begin{align*}
   \beta^*(\res^{\Sigma_n}_\Gamma(t_n))\  &= \ 
\sum_{[\sigma]\in \Gamma\backslash \Sigma_n/\Sigma_{n-1}} 
\beta^*\left(t_{\Gamma\cap{^\sigma \Sigma_{n-1}}}^\Gamma\right)\ = \ 
\sum_{[\sigma]\in \Gamma\backslash \Sigma_n/\Sigma_{n-1}} t_{\beta^{-1}(^\sigma \Sigma_{n-1})}^H
\end{align*}
in~$\mA(H)$. Transferring from $H$ to $G$ gives
\begin{align*}
  \tr_H^G(\beta^*(\res^{\Sigma_n}_\Gamma(t_n)))\ &= \    
  \sum_{[\sigma] \in \Gamma\backslash \Sigma_n/\Sigma_{n-1}} 
   t^H_{\beta^{-1}(^\sigma \Sigma_{n-1})} \text{\quad in~$\mA(G)$.}
\end{align*}
Since ${^\sigma \Sigma_{n-1}}$ has index $n$ in $\Sigma_n$,
the group $\Gamma\cap{^\sigma \Sigma_{n-1}}$ has index at most $n$ in $\Gamma$,
and hence the group $\beta^{-1}(^\sigma \Sigma_{n-1})=\beta^{-1}(\Gamma\cap{^\sigma \Sigma_{n-1}})$ 
has index at most $n$ in $H$.
So $I_n(G)$ is indeed contained in the group described in the statement
of the proposition.

For the other inclusion we consider a pair of closed subgroups $K\leq H$
in the ambient group~$G$ with $m=[H:K]\leq n$. 
A choice of bijection between $H/K$ and $\{1,\dots,m\}$ turns
the left translation action of $H$ on~$H/K$ into a homomorphism
$\beta:H\to\Sigma_m$ such that $H/K$ is isomorphic,
as an $H$-set, to $\beta^*(\{1,\dots,m\})$.
Since~$t_m\in I_m(\Sigma_m)\subset  I_n(\Sigma_m)$
and~$I_n$ is a global functor, we conclude that
\[ t_K^H  \ = \ \tr_H^G\left([H:K]\cdot 1- \tr^H_K(1)\right)\ = \ 
 \tr_H^G(\beta^*(t_m)) \ \in \ I_n(G)\ . \qedhere \]
\end{proof}

For every finite group $G$ the augmentation ideal $I(G)$ is generated
by the classes $t^G_H$ where~$H$ runs through all subgroups of $G$.
So Proposition~\ref{prop-describe I_n} shows that the filtration
by the subfunctors $I_n$ exhausts the augmentation ideal at the $|G|$-th stage:

\begin{cor} Let~$G$ be a finite group.
Then $I_n(G)=I(G)$ for~$n\geq |G|$.
\end{cor}

However, often the filtration stops earlier, for example for $p$-groups.

\begin{eg}[Finite $p$-groups]
Let $p$ be a prime and $P$ a finite $p$-group. 
Proposition~\ref{prop-describe I_n} shows that $I_n(P)=\{0\}$
for $n<p$. On the other hand, every proper subgroup $H$ of~$P$ admits
a sequence of intermediate subgroups
\[ H\ =\  H_0\ \subset\ H_1 \ \subset \ \dots \ \subset\  H_k\  =\ P  \]
such that $[H_i:H_{i-1}]=p$ for all $i=1,\dots, k$.
Then the class
\begin{align*}
t_H^P\ & = \   p^k \cdot 1 -  \tr^P_H(1)  \ = \  
 \sum_{i=1}^k p^{i-1}\cdot t_{H_{i-1}}^{H_i}  
\end{align*}
belongs to $I_p(P)$ by Proposition~\ref{prop-describe I_n}.
Since the classes $t_H^P$ generate the augmentation ideal, 
we conclude that $I_p(P)=I(P)$.
Hence the group~$\pi_0^P(Sp^n)$ is isomorphic to the Burnside ring~$\mA(P)$
for $1\leq n <p$, and free of rank~1 for $n\geq p$.
\end{eg}

We work out the symmetric product filtration on equivariant homotopy groups
for the symmetric groups~$\Sigma_k$ for~$k\leq 4$
and for the alternating group~$A_5$. 
The groups $G=\Sigma_4$ and~$G=A_5$ provide explicit examples 
of homotopy groups~$\pi_0^G(Sp^n)$ with non-trivial torsion.

\begin{eg}[Symmetric group $\Sigma_2$]\label{s2}
For the group~$\Sigma_2$ we have
$I_2(\Sigma_2) = I(\Sigma_2)$, freely generated by
the class $t_2$, i.e., the filtration terminates at the second step.
Hence the group~$\pi_0^{\Sigma_2}(\mS)$ is free of rank~2,
while the groups $\pi_0^{\Sigma_2}(Sp^n)$ are free of rank~1 for all $n\geq 2$.
\end{eg}

\begin{eg}[Symmetric group $\Sigma_3$]\label{s3}
The group~$\Sigma_3$ has four conjugacy classes of subgroups
with representatives $e, \Sigma_2, A_3$ and~$\Sigma_3$.
So the augmentation ideal $I(\Sigma_3)$ is free of rank~3, 
and a basis is given by the classes 
\[  t_3 \ = \ t_{\Sigma_2}^{\Sigma_3} \ , \quad
p^*(t_2)\ = \ t_{A_3}^{\Sigma_3}
 \text{\qquad and\qquad}  
\tr_{\Sigma_2}^{\Sigma_3}(t_2) \ = \ 2\cdot t_{\Sigma_2}^{\Sigma_3}\ - \  t_e^{\Sigma_3} \  ,\]
where $p:\Sigma_3\to\Sigma_2$ is the unique epimorphism.
Hence $I_2(\Sigma_3)$ is freely generated by the classes $p^*(t_2)$
and~$\tr_{\Sigma_2}^{\Sigma_3}(t_2)$, and $I_3(\Sigma_3) = I(\Sigma_3)$,
i.e., the filtration stabilizes at the third step.
Theorem~\ref{thm-pi_0 Sp^n} lets us conclude 
that the homotopy group~$\pi_0^{\Sigma_3}(Sp^n)$
is free for every~$n\geq 1$, and has rank~4 for $n=1$,
rank~2 for $n=2$,  and rank~1 for $n\geq 3$.
\end{eg}

\begin{eg}[Symmetric group $\Sigma_4$]\label{eg:S4}
The group~$\Sigma_4$ has 11 conjugacy classes of subgroups, displayed below;
the left column lists the order of a subgroup, and lines denote subconjugacy:
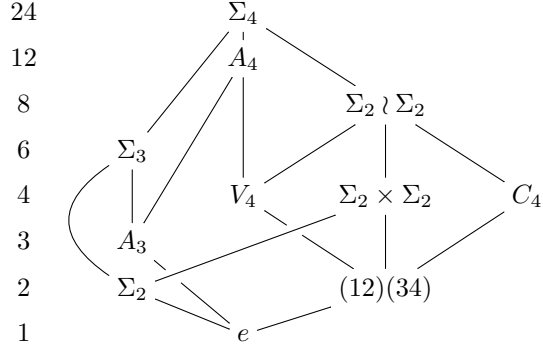
\[\xymatrix@R=1mm{
24 & &\Sigma_4\ar@{-}[dddl]\ar@{-}[d]\ar@{-}[ddr]\\
12 & &A_4\ar@{-}[ddd]\ar@{-}[ddddl]\\
 8 & &&\Sigma_2\wr\Sigma_2\ar@{-}[ddr]\ar@{-}[dd]\ar@{-}[ddl]\\
 6 & \Sigma_3\ar@{-}[dd]\ar@/_2pc/@{-}[ddd]\\
 4 & & V_4\ar@{-}[ddr]|(.36)\hole & \Sigma_2\times\Sigma_2\ar@{-}[ddll]\ar@{-}[dd] & 
C_4 \ar@{-}[ddl]\\
 3 & A_3\ar@{-}[ddr]|(.37)\hole\\
 2 &  \Sigma_2 \ar@{-}[rd] && (12)(34)\ar@{-}[dl]\\
 1 & & e
}\]
The augmentation ideal $I(\Sigma_4)$ is free of rank~10, and the classes
\begin{align}\label{eq:generators I_2(Sigma_4)}
 t_e^{\Sigma_2}\ ,  &\quad
t_{\Sigma_2}^{\Sigma_2\times\Sigma_2} \ , \quad
t_{(12)(34)}^{V_4} \ , \quad
t_{A_3}^{\Sigma_3} \ , \quad
t_{V_4}^{\Sigma_2\wr\Sigma_2} \ , \quad
t_{\Sigma_2\times\Sigma_2}^{\Sigma_2\wr\Sigma_2} \ , \quad
t_{C_4}^{\Sigma_2\wr\Sigma_2} \ , \quad
t_{A_4}^{\Sigma_4} 
\end{align}
together with the two classes
\[   t_{\Sigma_2\wr\Sigma_2}^{\Sigma_4}  \text{\qquad and\qquad}  
 t_{\Sigma_3}^{\Sigma_4}\ = \ t_4 \]
form a basis of $I(\Sigma_4)$. 

The group $I_2(\Sigma_4)$ is generated 
by the classes~$t_K^H$ as $(H,K)$ runs over all pairs of
nested subgroups with $[H:K]=2$. 
All classes of this particular form 
are linear combinations of the eight classes~\eqref{eq:generators I_2(Sigma_4)}:
\begin{align*}
t_{(12)(34)}^{\Sigma_2\times\Sigma_2}  \ &= \ t_{(12)(34)}^{V_4}\ +\  2\cdot t_{V_4}^{\Sigma_2\wr\Sigma_2}
\ -\ 2\cdot t_{\Sigma_2\times\Sigma_2}^{\Sigma_2\wr\Sigma_2} \\
 t_{(12)(34)}^{C_4} \ &= \  
t_{(12)(34)}^{V_4}\  +\ 2\cdot t_{V_4}^{\Sigma_2\wr\Sigma_2} 
\ -\ 2\cdot  t_{C_4}^{\Sigma_2\wr\Sigma_2}\\
 t_e^{(12)(34)} \ &= \  
\quad t_e^{\Sigma_2}\quad  +\ 2\cdot t_{\Sigma_2}^{\Sigma_2\times\Sigma_2} 
\ -\ 2\cdot  t_{(12)(34)}^{\Sigma_2\times\Sigma_2}  
\end{align*}
So the eight classes~\eqref{eq:generators I_2(Sigma_4)} form a basis of $I_2(\Sigma_4)$.

The group $I_3(\Sigma_4)$ is generated by the classes $t_K^H$ for all 
nested subgroup pairs with $[H:K]\leq 3$. 
We observe that
\begin{equation} \label{eq:honest_relation}
3\cdot t_4 \ = \ 3\cdot t_{\Sigma_3}^{\Sigma_4}  
\ = \ t_{\Sigma_2}^{\Sigma_2\times\Sigma_2} 
\ + \ 2\cdot t_{\Sigma_2\times\Sigma_2}^{\Sigma_2\wr\Sigma_2} 
\ + \ 4\cdot t_{\Sigma_2\wr\Sigma_2}^{\Sigma_4} -  t_{\Sigma_2}^{\Sigma_3} \ \in \ I_3(\Sigma_4)\ ;
\end{equation}
Proposition~\ref{prop-general inclusion} below explains in which way this relation is an
exceptional feature for $n=4$.
All classes~$t_K^H$ with~$[H:K]\leq 3$
are linear combinations of the classes~\eqref{eq:generators I_2(Sigma_4)}
and the two classes
\begin{equation}\label{eq:generators I_3(Sigma_4)}
  t_{\Sigma_2\wr\Sigma_2}^{\Sigma_4}  \text{\qquad and\qquad} 3\cdot t_{\Sigma_3}^{\Sigma_4}\ .
\end{equation}
Indeed:
\begin{align*}
 t_{V_4}^{A_4} \ &= \  
t_{V_4}^{\Sigma_2\wr\Sigma_2}\  +\  2\cdot t_{\Sigma_2\wr\Sigma_2}^{\Sigma_4} 
\ -\ 3\cdot  t_{A_4}^{\Sigma_4} \\
  t_{\Sigma_2}^{\Sigma_3} \ &= \ t_{\Sigma_2}^{\Sigma_2\times\Sigma_2} 
 +\  2\cdot t_{\Sigma_2\times\Sigma_2}^{\Sigma_2\wr\Sigma_2} 
 +\  4\cdot t_{\Sigma_2\wr\Sigma_2}^{\Sigma_4}\ -\ 3\cdot t_{\Sigma_3}^{\Sigma_4}  \\
t_e^{A_3}\ &= \quad t_e^{\Sigma_2}\quad + \ 2 \cdot t_{\Sigma_2}^{\Sigma_3} \quad - \ 3 \cdot t_{A_3}^{\Sigma_3}
\end{align*}
Since the eight elements~\eqref{eq:generators I_2(Sigma_4)}
and the two elements~\eqref{eq:generators I_3(Sigma_4)}
together are linearly independent, 
they form a basis of $I_3(\Sigma_4)$.

Since $[\Sigma_4:\Sigma_3]=4$, the last basis element
$t_4=t_{\Sigma_3}^{\Sigma_4}$ belongs to $I_4(\Sigma_4)$,
and thus $I_4(\Sigma_4)=I(\Sigma_4)$. The relation~\eqref{eq:honest_relation} shows that 
$I_3(\Sigma_4)$ has index~3 in $I_4(\Sigma_4)=I(\Sigma_4)$.
Altogether, Theorem~\ref{thm-pi_0 Sp^n} lets us conclude:
\begin{itemize}
\item  the group
$\pi_0^{\Sigma_4}(\mS)=\pi_0^{\Sigma_4}(Sp^1)$
is free of rank~11,
\item  the group
$\pi_0^{\Sigma_4}(Sp^2)$ is free of rank~3,
\item the group~$\pi_0^{\Sigma_4}(Sp^3)$ has rank~1 and its torsion subgroup has order~3, and
\item for all $n\geq 4$, the group~$\pi_0^{\Sigma_4}(Sp^n)$
is free of rank~1. 
\end{itemize}
\end{eg}

\begin{eg}[Alternating group $A_5$]\label{eg-A5}
The last example that we treat in detail is the alternating group~$A_5$.
The point is not just to have another explicit example, but we also
need the calculation of $I_5(A_5)$ in Example~\ref{eg:S5}
for identifying when the filtration of $\Sigma_5$ stabilizes.
The group $A_5$ has~9 conjugacy classes of subgroups:
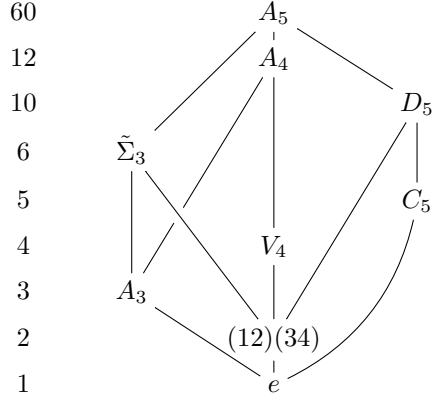
\[\xymatrix@R=1mm{
 60 &&  A_5 \ar@{-}[ddr]\ar@{-}[dddl]\ar@{-}[d]\\
 12 &&  A_4\ar@{-}[dddd]\ar@{-}[dddddl]|(.67)\hole & \\
 10 &&& D_5 \ar@{-}[dd]\ar@{-}[dddddl]\\  
  6 & \tilde\Sigma_3\ar@{-}[ddd]\ar@{-}[ddddr] && \\
  5 & & & C_5\ar@/^1pc/@{-}[ddddl] \\  
  4 & & V_4 \ar@{-}[dd] & & \\
  3 & A_3\ar@{-}[ddr]\\
  2 && (12)(34)\ar@{-}[d]\\
  1 && e
}\]
The group $\tilde\Sigma_3$ is generated by the elements $(1 2 3)$
and $(1 2)(4 5)$ and is isomorphic to~$\Sigma_3$
(but not conjugate in $\Sigma_5$ to the `standard' copy
of $\Sigma_3$~ generated by~$(1 2 3)$ and $(1 2)$).
The dihedral group $D_5$ is generated by the elements 
$(12345)$ and $(2 5)(3 4)$.

The augmentation ideal $I(A_5)$ is free of rank~8, 
and a convenient basis for our purposes is given by the classes
\begin{align}\label{eq:basis of I(A_5)}
 t_e^{(12)(34)}\ ,  &\quad
t_{(12)(34)}^{V_4} \ , \quad
t_{A_3}^{\tilde\Sigma_3} \ , \quad
t_{C_5}^{D_5} \ , \quad
t_{V_4}^{A_4} \ , \quad
t_{(12)(34)}^{\tilde\Sigma_3}+ t_{A_3}^{A_4}\ , \quad
t_{A_4}^{A_5} \text{\quad and\quad}  
 t_{D_5}^{A_5}\ . 
\end{align}
Proposition~\ref{prop-describe I_n} says that the group $I_2(A_5)$ 
is generated by the classes~$t_K^H$ as $(H,K)$ runs over all pairs of
nested subgroups with $[H:K]=2$, i.e, by the first four classes 
of the basis~\eqref{eq:basis of I(A_5)}.
So these four classes form a basis of $I_2(A_5)$.

We observe that
\begin{align}\label{eq-relation in I_3(A_5)}
3\cdot (  t_{(12)(34)}^{\tilde\Sigma_3} + t_{A_3}^{A_4})\ &= \ 
2\cdot t_{(12)(34)}^{V_4}\ +\ 4\cdot t_{V_4}^{A_4} 
\ +\ 3\cdot t_{A_3}^{\tilde\Sigma_3}\ + \ t_{(12)(34)}^{\tilde\Sigma_3} \ \in \ I_3(A_5)\ .
\end{align}
The group $I_3(A_5)$ is generated by the classes $t_K^H$ 
for all nested subgroup pairs with $[H:K]\leq 3$. Because
\begin{align*}
 t_{(12)(34)}^{\tilde\Sigma_3} \ &= \ 
3\cdot (  t_{(12)(34)}^{\tilde\Sigma_3} + t_{A_3}^{A_4})\ - \ 
2\cdot t_{(12)(34)}^{V_4}\ -\ 3\cdot t_{A_3}^{\tilde\Sigma_3}\ -\ 4\cdot t_{V_4}^{A_4} \\
 t_e^{A_3} \quad &= \ 
 t_e^{(12)(34)}\ - \ 3\cdot t_{A_3}^{\tilde\Sigma_3}\ +\ 
2\cdot t_{(12)(34)}^{\tilde\Sigma_3} 
\end{align*}
all such classes are linear combinations of the~6 classes
\[  t_e^{(12)(34)}\ ,  \quad
t_{(12)(34)}^{V_4} \ , \quad
t_{A_3}^{\tilde\Sigma_3} \ , \quad
t_{C_5}^{D_5}  \ , \quad
t_{V_4}^{A_4} \text{\qquad and\qquad} 3\cdot (  t_{(12)(34)}^{\tilde\Sigma_3} + t_{A_3}^{A_4}) \ .\]
Since these classes are linearly independent, 
they form a basis of the group~$I_3(A_5)$.

The group $I_4(A_5)$ is generated by the classes $t_K^H$
for all nested subgroups with $[H:K]\leq 4$. 
There is only one new generator, the class~$t_{A_3}^{A_4}$;
since $t_{(12)(34)}^{\tilde\Sigma_3}\in I_3(A_5)\subset I_4(A_5)$,
the group $I_4(A_5)$ is freely generated by 
the first six elements of the basis~\eqref{eq:basis of I(A_5)}.
Because $3\cdot t_{A_3}^{A_4}\in I_3(A_5)$,
the group $I_3(A_5)$ has index~3 in~$I_4(A_5)$.

The group $I_5(A_5)$ is generated by the classes $t_K^H$
for all nested subgroup pairs with $[H:K]\leq 5$. 
In particular, $I_5(A_5)$ contains the seventh element~$t_{A_4}^{A_5}$ 
of the basis~\eqref{eq:basis of I(A_5)}.
We observe that
\begin{equation}\label{eq:5D}
 5\cdot t_{D_5}^{A_5}\ = \ 
t_{(12)(34)}^{V_4}\ +\ 2\cdot t_{V_4}^{A_4}\ +\ 6 \cdot t_{A_4}^{A_5}\ -\  t_{(12)(34)}^{D_5}
\ \in\ I_5(A_5)\ .  
\end{equation}
All classes of the form $t_K^H$ with~$[H:K]\leq 5$
are linear combinations of the first seven classes 
of the basis~\eqref{eq:basis of I(A_5)} and the class~\eqref{eq:5D}:
\begin{align*}
 t_{(12)(34)}^{D_5}\ &= \ 
t_{(12)(34)}^{V_4}\ +\ 2\cdot t_{V_4}^{A_4}\ +\  6 \cdot t_{A_4}^{A_5}\ -\ 
 5\cdot t_{D_5}^{A_5}\\
t_e^{C_5}\quad &= \ t_e^{(12)(34)}\ + \ 2\cdot t_{(12)(34)}^{V_4}\ - \ 5\cdot t_{C_5}^{D_5}
\ + \ 4\cdot t_{V_4}^{A_4}\ + \ 12 \cdot t_{A_4}^{A_5} \ - \ 2\cdot (5\cdot t_{D_5}^{A_5})  \ .
\end{align*}
The group $I_5(A_5)$ is thus generated by the linearly independent classes
\[  t_e^{(12)(34)}\ ,  \quad
t_{(12)(34)}^{V_4} \ , \quad
t_{A_3}^{\tilde\Sigma_3} \ , \quad
t_{C_5}^{D_5}  \ , \quad
t_{V_4}^{A_4} \ , \quad
t_{(12)(34)}^{\tilde\Sigma_3} + t_{A_3}^{A_4}\ , \quad 
t_{A_4}^{A_5}
\text{\qquad and\qquad}
5\cdot t_{D_5}^{A_5}\ . \]
So $I_5(A_5)$ has full rank~8, but index~5 in the augmentation ideal~$I(A_5)$.
Since $[A_5:D_5]=6$, the last basis element $ t_{D_5}^{A_5}$
belongs to~$I_6(A_5)$, and we conclude that $I_6(A_5)=I(A_5)$ 
is the full augmentation ideal.
Altogether, Theorem~\ref{thm-pi_0 Sp^n} lets us conclude that 
\begin{itemize}
\item  the group
$\pi_0^{A_5}(\mS)=\pi_0^{A_5}(Sp^1)$
is free of rank~9,
\item  the group
$\pi_0^{A_5}(Sp^2)$ is free of rank~5,
\item the group~$\pi_0^{A_5}(Sp^3)$ has rank~3 and its torsion subgroup has order~3, 
\item the group~$\pi_0^{A_5}(Sp^4)$ is free of rank~3,
\item the group~$\pi_0^{A_5}(Sp^5)$ has rank~1 and its torsion subgroup has order~5,  and
\item for all $n\geq 6$, the group~$\pi_0^{A_5}(Sp^n)$
is free of rank~1. 
\end{itemize}
\end{eg}

\begin{eg}[Symmetric group $\Sigma_5$]\label{eg:S5}
We refrain from a complete calculation of the groups~$\pi_0^{\Sigma_5}(Sp^n)$, 
but we work out where the filtration for $\Sigma_5$ stabilizes.
The previous examples could be mistaken as evidence that
the group $I_n(\Sigma_n)$ coincides with the full augmentation ideal $I(\Sigma_n)$
for every $n$; equivalently, one could get the false impression that 
the group $\pi_0^{\Sigma_n}(Sp^n)$ is always free of rank~1. 
While this is true for $n\leq 4$, we will now see that it fails for $n=5$,
i.e., that $I_5(\Sigma_5)$ is strictly smaller than~$I(\Sigma_5)$.

We let~$B$ denote the subgroup of~$\Sigma_5$
generated by the elements $(12345)$ and $(2 3 5 4)$;
this group has order~20 and is isomorphic to 
the semi-direct product $\mF_5\rtimes (\mF_5)^\times$, the affine linear
group of the field~$\mF_5$.
The intersection of~$B$ with the alternating group~$A_5$
is the dihedral group $D_5$.
The double coset formula thus gives 
\begin{align*}
 \res^{\Sigma_5}_{A_5}(t_B^{\Sigma_5}) \ &= \   
6\ - \ \res^{\Sigma_5}_{A_5}(\tr_B^{\Sigma_5}(1)) \ = \   
6\ - \  \tr_{D_5}^{A_5}(\res^B_{D_5}(1))  \ = \  t_{D_5}^{A_5}\ . 
\end{align*}
We showed in the previous Example~\ref{eg-A5} that the class $t_{D_5}^{A_5}$
does {\em not} belong to $I_5(A_5)$. 
Since~$I_5$ is closed under restriction maps,
the class $t_B^{\Sigma_5}$ does {\em not} belong to $I_5(\Sigma_5)$,
and hence $I_5(\Sigma_5)\ne I(\Sigma_5)$.

Every subgroup~$H$ of~$\Sigma_5$ admits a nested sequence of subgroups
\[ H \ = \ H_0 \ \subset \ H_1 \ \subset \ \dots \ \subset \ H_k  \]
with $[H_i:H_{i-1}]\leq 6$ for all $i=1,\dots,k$
and such that the last group $H_k$ is either the full group $\Sigma_5$ or 
conjugate to the maximal subgroup $\Sigma_3\times\Sigma_2$ of index~10.
The relation
\begin{align*}
t^{\Sigma_5}_{\Sigma_3\times\Sigma_2}\ &= \ 
t^{\Sigma_3\times\Sigma_2}_{\Sigma_3}\ - \ 
t^{\Sigma_3\times\Sigma_2}_{\Sigma_2\times\Sigma_2}\ - \ 
t^{\Sigma_4}_{\Sigma_3}\  + \ 
t^{\Sigma_2\wr\Sigma_2}_{\Sigma_2\times\Sigma_2}\ + \ 2\cdot t^{\Sigma_4}_{\Sigma_2\wr\Sigma_2}\ + \ 
2\cdot t^{\Sigma_5}_{\Sigma_4}
\end{align*}
shows that the class $t^{\Sigma_5}_{\Sigma_3\times\Sigma_2}$ lies in $I_5(\Sigma_5)$.
So $t_H^{\Sigma_n}$ belongs to~$I_6(\Sigma_5)$ for every subgroup~$H$
of~$\Sigma_n$, and hence $I_6(\Sigma_5) =I(\Sigma_5)$.
\end{eg}

We pause to point out a curious phenomenon that happens only for $n=4$.
This exceptional behavior can be traced back to the fact that the
alternating group~$A_4$ has a subgroup of `unusually small index'
(the Klein group~$V_4$ of index~3),
compare the proof of the next proposition.
We have seen in~\eqref{eq:honest_relation} that
$3\cdot t_4$ lies in $I_3(\Sigma_4)$.
Since the class $t_4$ generates the global functor $I_4$, this implies that
\[ 3\cdot I_4 \ \subset \ I_3 \ \subset \ I_4 \ . \]
So after inverting~3, the inclusion $I_3 \to I_4$ and the epimorphism of global functors
\[  \upi_0(Sp^3)\ \to \ \upi_0(Sp^4) \]
induced by the inclusion $Sp^3\to Sp^4$ both become isomorphisms. However:

\begin{prop}\label{prop-general inclusion} 
For every $n\geq 2$ with $n\ne 4$, 
the inclusion $I_{n-1}\to I_n$ is not a rational isomorphism.
\end{prop}
\begin{proof}
Example~\ref{s2} shows that no non-zero multiple of the class~$t_2$ 
belongs to~$I_1(\Sigma_2)$.
Example~\ref{s3} shows that no non-zero multiple of~$t_3$ 
belongs to~$I_2(\Sigma_3)$. So we assume $n\geq 5$ for the rest of the argument.
We recall that the alternating group~$A_n$ has no proper subgroup~$H$ 
of index less than~$n$.
Indeed, the left translation action on $A_n/H$ 
provides a homomorphism~$\rho:A_n\to\Sigma(A_n/H)$ to the symmetric group
of the underlying set of~$A_n/H$. For $[A_n:H]<n$, the order of
$\Sigma(A_n/H)$ is strictly less than the order of~$A_n$. So the homomorphism
$\rho$ has a non-trivial kernel. Since the group $A_n$ is simple, 
$\rho$ must be trivial, which forces~$H=A_n$.

Now we prove the proposition.
The class $t_{A_{n-1}}^{A_n}=\res^{\Sigma_n}_{A_n}(t_n)$ 
belongs to $I_n(A_n)$, but for~$n> 4$ 
no non-zero multiple of it belongs to $I_{n-1}(A_n)$.
Indeed, otherwise Proposition~\ref{prop-describe I_n} would allow us to write
\begin{align*}
 k\cdot t_{A_{n-1}}^{A_n}\ &=  \ \lambda_1\cdot t_{K_1}^{H_1} +\dots +\lambda_m\cdot t_{K_m}^{H_m} 
\end{align*}
in~$\mA(A_n)$, for certain integers $k,\lambda_1,\dots,\lambda_m$ and 
nested subgroup pairs with $1< [H_i:K_i]< n$.
Since~$A_n$ has no proper subgroup of index less than~$n$,
the groups $H_1,\dots,H_m$ must all be different from the full group~$A_n$. 
We expand both sides in terms of 
the basis of~$\mA(A_n)$ given by the classes $\tr_H^{A_n}(1)$
(for $H$ running through the conjugacy classes of subgroups).
On the right hand side the coefficient of the basis element $\tr_{A_n}^{A_n}(1)=1$ is zero,
whereas the coefficient on the left hand side is~$k n$. So we must have~$k=0$.
\end{proof}

We conclude by looking more closely at the limit case, 
the infinite symmetric product spectrum.
We remark without proof that, generalizing the non-equivariant situation, 
the orthogonal spectrum $Sp^\infty$
is globally equivalent to the orthogonal spectrum $H\mZ$ defined by
\[ (H\mZ)(V) \ = \ \mZ[S^V] \ , \]
the reduced free abelian group generated by the $n$-sphere.
Theorem~\ref{thm-pi_0 Sp^n} shows that
\[ \mA/I_\infty \ \iso \ \upi_0(Sp^\infty) \ ,\]  
induced by the action of~$\mA$ on the class $i_*(1)$.
For every compact Lie group~$G$, the map 
\[ \res^G_e \ : \ \pi_0^G(Sp^\infty) \ \to \ \pi_0^e(Sp^\infty) \iso\mZ  \]
is a split epimorphism, so the group $\pi_0^G(Sp^\infty)$
is free of rank~1 if and only if $I_\infty(G)=I(G)$.

We can split the group $\mA(G)/I_{\infty}(G)$,
and hence the group $\pi_0^G(Sp^\infty)$, into summands 
indexed by conjugacy classes of connected subgroups of~$G$.
If $C$ is such a connected subgroup, we denote by
$\mA(G;C)$ the subgroup of the Burnside ring~$\mA(G)$ that is
generated by the transfers $\tr_H^G(1)$ for all subgroups~$H$ 
with~$C=H^\circ$, the path component of the identity of~$H$ 
(or equivalently, $H$ contains $C$ as a finite index subgroup).
Then
\[ \mA(G)\ = \ \bigoplus_{(C)} \, \mA(G;C)\]
where the sum runs over conjugacy classes of connected subgroups of~$G$.
Proposition~\ref{prop-describe I_n} shows that~$I_\infty(G)$
is generated as an abelian group by the classes
\[ t^H_K\ = \ [H:K]\cdot \tr_H^G(1) - \tr_K^G(1) \ \in \ \mA(G)\]
as $(H,K)$ runs through all pairs of nested closed subgroups
such that $K$ has finite index in~$H$,
and $H$ has finite Weyl group in $G$.
Then~$K$ and~$H$ have the same connected component of
the identity, i.e., $K^\circ=H^\circ$, so the relation $t^H_K$ belongs
to the direct summand $\mA(G;K^\circ)$. 
Hence
\[ \pi_0^G(Sp^\infty)\ \iso \ \mA(G)/I_\infty(G)\ = \ 
\bigoplus_{(C)}\,\left( \mA(G;C)/I_\infty(G;C)\right) \ ,\]
where $I_\infty(G;C)$ is the subgroup of $\mA(G;C)$ generated by
the classes $t^H_K$ with $H^\circ=K^\circ=C$.
The summands behave quite differently according to whether $C$ has
infinite or finite Weyl group:
\begin{itemize}
\item If~$C$ has an infinite Weyl group, then for every subgroup $H\leq G$ with 
$H^\circ=C$ the class $[H:C]\cdot \tr_H^G(1)$ belongs to~$I_\infty(G;C)$.
So the class $\tr_H^G(1)$ becomes torsion in 
the quotient group $\mA(G;C)/I_\infty(G;C)$, which is thus a torsion group.

\item If~$C$ has finite Weyl group, 
and~$H\leq G$ satisfies $H^\circ=C$, then the relations
\[   C = H^\circ \ \leq\  H\ \leq \ N_G H \ \leq \ N_G C    \]
show that~$H$ has finite Weyl group and finite index
in~$N_G C$. So
\[ t_H^{N_G C} \ = \ [N_G C:H]\cdot \tr_{N_G C}^G(1)\ - \ \tr_H^G(1)\ \in \ I_\infty(G;C)\]
and in the quotient group  $\mA(G;C)/I_\infty(G;C)$, the class
$\tr_H^G(1)$ becomes a multiple of the class~$\tr_{N_G C}^G(1)$.
Hence the group $\mA(G;C)/I_\infty(G;C)$ is free of rank~1, 
generated by~$\tr_{N_G C}^G(1)$.

In the situation at hand, the subgroup~$C$ 
can be recovered as the identity component of its normalizer.
A compact Lie group has only finitely many conjugacy classes of  
subgroups that are normalizers of connected subgroups, 
see~\cite[VII Lemma 3.2]{borel:bredon}.
So there are only
finitely many conjugacy classes of connected subgroups with finite Weyl group.
\end{itemize}

So altogether we conclude that the group 
$\pi_0^G(Sp^\infty)$ is a direct sum of a torsion group
and a free abelian group of finite rank.
In particular, the rationalization
$\mQ\tensor \pi_0^G(Sp^\infty)$ is a finite dimensional~$\mQ$-vector space
with basis consisting of the classes $\tr_C^G(1)$ as $C$ runs through 
the conjugacy classes of connected subgroups of~$G$ with finite Weyl group.
Unfortunately, the author does not know an example when the 
torsion subgroup of $\pi_0^G(Sp^\infty)$ is non-trivial.

\begin{eg}\label{eq-SU(2)}
If every subgroup~$H$ with finite Weyl group
also has finite index in~$G$, then $I_\infty(G)=I(G)$ and~$\pi_0^G(Sp^\infty)$
is free of rank~1. This holds, for example, when $G$ is finite or a torus. 

An example for which  $\pi_0^G( Sp^{\infty})$ has rank bigger than~1 is $G=S U(2)$.
Here there are three conjugacy classes of connected subgroups:
the trivial subgroup, the conjugacy class of the maximal tori
and the full group $S U(2)$. 
Among these, the maximal tori and $S U(2)$ have finite Weyl groups,
so the classes~$1$ and~$\tr_N^{S U(2)}(1)$ are a $\mZ$-basis for
$\pi_0^{S U(2)}( Sp^\infty )$ modulo torsion, where $N$ is a
maximal torus normalizer.
\end{eg}

\end{document}